\documentclass[12pt,reqno]{smfart}
\usepackage{amsmath,amssymb,smfthm,stmaryrd,epic,enumerate}
\usepackage[frenchb]{babel}
\usepackage[T1]{fontenc}

\usepackage[all]{xy}
\usepackage{a4wide}

\usepackage{mathrsfs}
\usepackage[active]{srcltx}

\usepackage{lmodern}

\usepackage{hyperref}

\NumberTheoremsIn{subsection}\SwapTheoremNumbers

\begin{document}
\newtheorem{prop-defi}[smfthm]{Proposition-DÈfinition}
\newtheorem{notas}[smfthm]{Notations}
\newtheorem{nota}[smfthm]{Notation}
\newtheorem{defis}[smfthm]{DÈfinitions}
\newtheorem{hypo}[smfthm]{HypothËse}
\newtheorem*{theo*}{ThÈorËme}
\newtheorem*{hyp*}{HypothËses}

\def\Tm{{\mathbb T}}
\def\Um{{\mathbb U}}
\def\Am{{\mathbb A}}
\def\Fm{{\mathbb F}}
\def\Mm{{\mathbb M}}
\def\Nm{{\mathbb N}}
\def\Pm{{\mathbb P}}
\def\Qm{{\mathbb Q}}
\def\Zm{{\mathbb Z}}
\def\Dm{{\mathbb D}}
\def\Cm{{\mathbb C}}
\def\Rm{{\mathbb R}}
\def\Gm{{\mathbb G}}
\def\Lm{{\mathbb L}}
\def\Km{{\mathbb K}}
\def\Om{{\mathbb O}}
\def\Em{{\mathbb E}}
\def\Xm{{\mathbb X}}

\def\BC{{\mathcal B}}
\def\QC{{\mathcal Q}}
\def\TC{{\mathcal T}}
\def\ZC{{\mathcal Z}}
\def\AC{{\mathcal A}}
\def\CC{{\mathcal C}}
\def\DC{{\mathcal D}}
\def\EC{{\mathcal E}}
\def\FC{{\mathcal F}}
\def\GC{{\mathcal G}}
\def\HC{{\mathcal H}}
\def\IC{{\mathcal I}}
\def\JC{{\mathcal J}}
\def\KC{{\mathcal K}}
\def\LC{{\mathcal L}}
\def\MC{{\mathcal M}}
\def\NC{{\mathcal N}}
\def\OC{{\mathcal O}}
\def\PC{{\mathcal P}}
\def\UC{{\mathcal U}}
\def\VC{{\mathcal V}}
\def\XC{{\mathcal X}}
\def\SC{{\mathcal S}}

\def\BF{{\mathfrak B}}
\def\AF{{\mathfrak A}}
\def\GF{{\mathfrak G}}
\def\EF{{\mathfrak E}}
\def\CF{{\mathfrak C}}
\def\DF{{\mathfrak D}}
\def\JF{{\mathfrak J}}
\def\LF{{\mathfrak L}}
\def\MF{{\mathfrak M}}
\def\NF{{\mathfrak N}}
\def\XF{{\mathfrak X}}
\def\UF{{\mathfrak U}}
\def\KF{{\mathfrak K}}
\def\FF{{\mathfrak F}}

\def \longmapright#1{\smash{\mathop{\longrightarrow}\limits^{#1}}}
\def \mapright#1{\smash{\mathop{\rightarrow}\limits^{#1}}}
\def \lexp#1#2{\kern \scriptspace \vphantom{#2}^{#1}\kern-\scriptspace#2}
\def \linf#1#2{\kern \scriptspace \vphantom{#2}_{#1}\kern-\scriptspace#2}
\def \linexp#1#2#3 {\kern \scriptspace{#3}_{#1}^{#2} \kern-\scriptspace #3}

\def \Ext{\mathop{\mathrm{Ext}}\nolimits}
\def \ad{\mathop{\mathrm{ad}}\nolimits}
\def \sh{\mathop{\mathrm{Sh}}\nolimits}
\def \irr{\mathop{\mathrm{Irr}}\nolimits}
\def \FH{\mathop{\mathrm{FH}}\nolimits}
\def \FPH{\mathop{\mathrm{FPH}}\nolimits}
\def \coh{\mathop{\mathrm{Coh}}\nolimits}
\def \res{\mathop{\mathrm{Res}}\nolimits}
\def \op{\mathop{\mathrm{op}}\nolimits}
\def \rec {\mathop{\mathrm{rec}}\nolimits}
\def \art{\mathop{\mathrm{Art}}\nolimits}
\def \vol {\mathop{\mathrm{vol}}\nolimits}
\def \cusp {\mathop{\mathrm{Cusp}}\nolimits}
\def \scusp {\mathop{\mathrm{Scusp}}\nolimits}
\def \Iw {\mathop{\mathrm{Iw}}\nolimits}
\def \JL {\mathop{\mathrm{JL}}\nolimits}
\def \speh {\mathop{\mathrm{Speh}}\nolimits}
\def \isom {\mathop{\mathrm{Isom}}\nolimits}
\def \Vect {\mathop{\mathrm{Vect}}\nolimits}
\def \groth {\mathop{\mathrm{Groth}}\nolimits}
\def \hom {\mathop{\mathrm{Hom}}\nolimits}
\def \deg {\mathop{\mathrm{deg}}\nolimits}
\def \val {\mathop{\mathrm{val}}\nolimits}
\def \det {\mathop{\mathrm{det}}\nolimits}
\def \rep {\mathop{\mathrm{Rep}}\nolimits}
\def \spec {\mathop{\mathrm{Spec}}\nolimits}
\def \fr {\mathop{\mathrm{Fr}}\nolimits}
\def \frob {\mathop{\mathrm{Frob}}\nolimits}
\def \ker {\mathop{\mathrm{Ker}}\nolimits}
\def \im {\mathop{\mathrm{Im}}\nolimits}
\def \Red {\mathop{\mathrm{Red}}\nolimits}
\def \red {\mathop{\mathrm{red}}\nolimits}
\def \aut {\mathop{\mathrm{Aut}}\nolimits}
\def \diag {\mathop{\mathrm{diag}}\nolimits}
\def \spf {\mathop{\mathrm{Spf}}\nolimits}
\def \Def {\mathop{\mathrm{Def}}\nolimits}
\def \twist {\mathop{\mathrm{Twist}}\nolimits}
\def \supp {\mathop{\mathrm{Supp}}\nolimits}
\def \Id {{\mathop{\mathrm{Id}}\nolimits}}
\def \lie {{\mathop{\mathrm{Lie}}\nolimits}}
\def \Ind{\mathop{\mathrm{Ind}}\nolimits}
\def \ind {\mathop{\mathrm{ind}}\nolimits}
\def \bad {\mathop{\mathrm{Bad}}\nolimits}
\def \top {\mathop{\mathrm{Top}}\nolimits}
\def \ker {\mathop{\mathrm{Ker}}\nolimits}
\def \coker {\mathop{\mathrm{Coker}}\nolimits}
\def \gal {{\mathop{\mathrm{Gal}}\nolimits}}
\def \Nr {{\mathop{\mathrm{Nr}}\nolimits}}
\def \nrd {{\mathop{\mathrm{Nrd}}\nolimits}}
\def \rn {{\mathop{\mathrm{rn}}\nolimits}}
\def \tr {{\mathop{\mathrm{Tr~}}\nolimits}}
\def \Sp {{\mathop{\mathrm{Sp}}\nolimits}}
\def \st {{\mathop{\mathrm{St}}\nolimits}}
\def \sp{{\mathop{\mathrm{Sp}}\nolimits}}
\def \perv{\mathop{\mathrm{Perv}}\nolimits}
\def \tor {{\mathop{\mathrm{Tor}}\nolimits}}
\def \gr {{\mathop{\mathrm{gr}}\nolimits}}
\def \nilp {{\mathop{\mathrm{Nilp}}\nolimits}}
\def \obj {{\mathop{\mathrm{Obj}}\nolimits}}
\def \Spl {{\mathop{\mathrm{Spl}}\nolimits}}

\def \rem{{\noindent\textit{Remarque:~}}}
\def \rems{{\noindent\textit{Remarques:~}}}
\def \ext {{\mathop{\mathrm{Ext}}\nolimits}}
\def \End {{\mathop{\mathrm{End}}\nolimits}}

\def\semi{\mathrel{>\!\!\!\triangleleft}}
\let \DS=\displaystyle
\def\HT{{\mathop{\mathcal{HT}}\nolimits}}

\def \hi{\HC}
\newcommand*{\tarrow}{\relbar\joinrel\mid\joinrel\twoheadrightarrow}
\newcommand*{\harrow}{\lhook\joinrel\relbar\joinrel\mid\joinrel\rightarrow}
\newcommand*{\rarrow}{\relbar\joinrel\mid\joinrel\rightarrow}
\def \coim {{\mathop{\mathrm{Coim}}\nolimits}}
\def \can {{\mathop{\mathrm{can}}\nolimits}}
\def\LFF{{\mathscr L}}

\setcounter{secnumdepth}{3} \setcounter{tocdepth}{3}

\def \Fil{\mathop{\mathrm{Fil}}\nolimits}
\def \CoFil{\mathop{\mathrm{CoFil}}\nolimits}
\def \Fill{\mathop{\mathrm{Fill}}\nolimits}
\def \CoFill{\mathop{\mathrm{CoFill}}\nolimits}
\def\SF{{\mathfrak S}}
\def\PF{{\mathfrak P}}
\def \EFil{\mathop{\mathrm{EFil}}\nolimits}
\def \EFill{\mathop{\mathrm{EFill}}\nolimits}
\def \FP{\mathop{\mathrm{FP}}\nolimits}

\let \longto=\longrightarrow
\let \oo=\infty

\let \d=\delta
\let \k=\kappa

\newcommand{\marque}{\addtocounter{smfthm}{1}
{\smallskip \noindent \textit{\thesmfthm}~---~}}

\renewcommand\atop[2]{\ensuremath{\genfrac..{0pt}{1}{#1}{#2}}}

\title[Congruences automorphes et classes de cohomologie de torsion]{Congruences automorphes et torsion dans la cohomologie d'une variÈtÈ de Shimura unitaire simple}

\alttitle{Automorphic congruences and torsion in the cohomology of a simple unitary Shimura variety}

\author{Boyer Pascal}
\email{boyer@math.univ-paris13.fr}
\address{UniversitÈ Paris 13, Sorbonne Paris CitÈ \\
LAGA, CNRS, UMR 7539\\ 
F-93430, Villetaneuse (France) \\
PerCoLaTor: ANR-14-CE25}

\thanks{L'auteur remercie l'ANR pour son soutien dans le cadre du projet PerCoLaTor 14-CE25.}

\frontmatter

\begin{abstract}
Nous donnons d'une part un procÈdÈ relativement souple de construction de classes de cohomologie de 
torsion dans la cohomologie d'une variÈtÈ 
de Shimura de type Kottwitz-Harris-Taylor ‡ coefficients dans un systËme local \og pas trop rÈgulier \fg.
D'autre part nous montrons qu'associÈe ‡ toute classe de cohomologie de torsion, il existe une infinitÈ 
de reprÈsentations irrÈductibles automorphes cohomologiques,
en caractÈristique nulle, qui sont deux ‡ deux non isomorphes et
faiblement congruentes au sens de \cite{vigneras-langlands}.
\end{abstract}

\begin{altabstract}
We first give a relative flexible process to construct torsion cohomology classes for Shimura varieties of 
Kottwitz-Harris-Taylor type
with coefficient in a non too regular local system. We then prove that associated to each torsion
cohomology class, there exists a infinity of irreducible automorphic representations in characteristic zero, 
which are pairwise non isomorphic and weakly congruent in the sense of \cite{vigneras-langlands}.

\end{altabstract}

\subjclass{11F70, 11F80, 11F85, 11G18, 20C08}

\keywords{VariÈtÈs de Shimura, cohomologie de torsion, idÈal maximal de l'algËbre de Hecke, 
localisation de la cohomologie, reprÈsentation galoisienne}

\altkeywords{Shimura varieties, torsion in the cohomology, maximal ideal of the Hecke algebra,
localized cohomology, galois representation}

\maketitle

\pagestyle{headings} \pagenumbering{arabic}

\tableofcontents
%
%
\renewcommand{\theequation}{\arabic{section}.\arabic{subsection}.\arabic{smfthm}}

\section*{Introduction}
\renewcommand{\thesmfthm}{\arabic{section}.\arabic{subsection}.\arabic{smfthm}}

Soient $F=EF^+$ un corps CM et $X_{I,\eta} \rightarrow \spec F$ une variÈtÈ de Shimura de type 
Kottwitz-Harris-Taylor associÈe donc ‡ une groupe de similitudes $G/\Qm$ et un sous-groupe compact
ouvert $I$. On note $\Spl(I)$ l'ensemble des nombres premiers $p \neq l$ tels que 
\begin{itemize}
\item $p=uu^c$ est dÈcomposÈ dans l'extension quadratique imaginaire $E/Q$,

\item $G(\Qm_p)$ est dÈcomposÈ, i.e. de la forme $\Qm_p^\times \times \prod_{v |u} GL_d(F_v)$ et

\item la composante locale en $p$ de $I$ est maximale.
\end{itemize}
Pour $l$ un nombre premier, et $\xi$ une reprÈsentation irrÈductible algÈbrique de $G$ donnant lieu
‡ un $\overline \Zm_l$-systËme local $V_\xi$ sur $X_{I,\bar \eta}$, dans \cite{boyer-imj} on Ètudie la 
torsion dans les $H^i(X_{I,\bar \eta},V_\xi)$
et on montre en particulier que d'un point de vue automorphe, rien de nouveau apparait au sens o˘
toutes les classes de torsion se relËvent en caractÈristique nulle dans la cohomologie en degrÈ mÈdian
d'une variÈtÈ d'Igusa. D'aprËs \cite{scholze-cara}, le phÈnomËne semble partagÈ par une large classe de
variÈtÈs de Shimura. Ainsi au moins du point de vue de la correspondance
de Langlands, l'existence de classes de cohomologie de torsion dans la cohomologie
des variÈtÈs de Shimura n'est d'aucun intÈrÍt. 

Le but de ce travail, qui prolonge \cite{boyer-imj}, est alors 
dans un premier temps, de donner un procÈdÈ \og relativement souple \fg{} de construction de 
classes de cohomologie de torsion pour ces variÈtÈs de Shimura. Pour ce faire, on choisit
une reprÈsentation $\xi$ pas trop rÈguliËre au sens de \cite{lan-suh} de faÁon ‡ ce que
la $\overline \Qm_l$-cohomologie de $X_{I,\bar \eta}$ ‡ coefficients dans $V_\xi$,
ne soit pas concentrÈe en degrÈ mÈdian.
Le principe est alors de calculer la cohomologie via la suite spectrale des cycles Èvanescents en
une place $v$ divisant suffisamment le niveau, puis d'utiliser les filtrations du $\overline \Zm_l$-faisceau 
pervers des 
cycles Èvanescents $\Psi_{I,v}$ construites dans \cite{boyer-torsion} et dont les graduÈs se dÈcrivent
‡ l'aide des faisceaux pervers d'Harris-Taylor. Ces derniers sont en particulier indexÈs par une
reprÈsentation irrÈductible cuspidale de $GL_g(F_v)$ avec $1 \leq g \leq d$. Comme dÈj‡ remarquÈ
dans \cite{boyer-aif}, lorsque la rÈduction modulo $l$ d'un tel $\pi_v$ n'est plus supercuspidale,
la cohomologie du faisceau pervers associÈ admet de la torsion pourvu que le niveau
soit suffisamment grand, cf. la proposition \ref{prop-torsion0}.
Tout l'objet du \S \ref{para-torsion-coho} est alors de montrer que cette torsion subsiste dans 
l'aboutissement de la suite spectrale associÈe ‡ la filtration de $\Psi_{I,v}$. Voici une version
imprÈcise du rÈsultat principal obtenu dans cette direction, cf. \ref{theo-equivalence2}.

\begin{theo*} Supposons qu'il existe 
\begin{itemize}
\item une $\overline \Fm_l$-reprÈsentation irrÈductible supercuspidale
$\varrho$ de $GL_g(F_v)$ dont la droite de Zelevinsky est de cardinal $2$, 

\item un $\overline \Zm_l$-relËvement $\pi_v$ ainsi 

\item qu'une reprÈsentation $\Pi$ automorphe irrÈductible et $\xi$-cohomologique, 
dont la composante locale en $v$ est de la forme $\speh_s (\pi_v) \times ?$
avec $s=\lfloor \frac{d}{g} \rfloor \geq 4$ et $?$ une reprÈsentation quelconque.
\end{itemize}
Alors pour un niveau fini $I$ tel que $\Pi$ possËde des vecteurs non nuls invariants sous $I$,
la torsion de la cohomologie de $X_{I,\bar \eta}$ ‡ coefficients dans $V_\xi$, est non
nulle.
\end{theo*}

Dans un deuxiËme temps au \S \ref{para-principal}, on propose une application automorphe
‡ l'existence de classes de torsion dans la cohomologie des variÈtÈs de Shimura de type
Kottwitz-Harris-Taylor, en construisant des congruences automorphes faibles au sens du \S 3 de 
\cite{vigneras-langlands}. Pour formuler une version allÈgÈe du rÈsultat principal, cf. 
le corollaire \ref{coro-principal}, notons suivant \ref{nota-spl}, $\Spl(I)$ l'ensemble des places
de $F$ au dessus d'un nombre premier
$p \neq l$ dÈcomposÈ dans $E$ et telle que $B_v^\times \simeq GL_d(F_v)$

\begin{theo*}
Pour $\mathfrak m$ l'idÈal maximal d'une algËbre de Hecke
non ramifiÈe, associÈ ‡ une classe de torsion dans la cohomologie de $X_{I,\bar \eta}$
‡ coefficient dans $V_\xi$, il existe une famille 
$$\Pi(p)$$ 
indexÈe par
les nombres premiers $p \in \Spl (I)$, de reprÈsentations irrÈductibles automorphes 
$\xi$-cohomologiques, telle que
pour premier $q \in \Spl(I)$ distinct de $p$, la composante locale en $q$ (resp. en $p$) 
de $\Pi(p)$ est non ramifiÈe (resp. est ramifiÈe), ses paramËtres de Satake modulo $l$ Ètant donnÈs
par $\mathfrak m$. 
\end{theo*}

En particulier pour $p \neq q$ des premiers de $\Spl(I)$, 
les reprÈsentations $\Pi(p)$ et $\Pi(q)$ ne sont pas isomorphes alors qu'elles sont non ramifiÈes en
toute place de $\Spl(I)- \{ p,q \}$ et y partagent les mÍmes paramËtres de Satake modulo $l$.
Au sens du \S 3 de \cite{vigneras-langlands}, on dit que ces reprÈsentations sont
faiblement congruentes.

\section{Rappels}

\subsection{sur les reprÈsentations}

ConsidÈrons un corps local $K$ muni de sa valeur absolue $| - |$: on note $q$ le cardinal
de son corps rÈsiduel. Pour $\pi$ une reprÈsentation de $GL_d(K)$ et $n \in \frac{1}{2} \Zm$, on note 
$$\pi \{ n \}:= \pi \otimes q^{-n \val \circ \det}.$$

\begin{notas}
Pour $\pi_1$ et $\pi_2$ des reprÈsentations de respectivement $GL_{n_1}(K)$ et
$GL_{n_2}(K)$, $\pi_1 \times \pi_2$ dÈsigne l'induite parabolique normalisÈe
$$\pi_1 \times \pi_2:=\ind_{P_{n_1,n_1+n_2}(K)}^{GL_{n_1+n_2}(K)}
\pi_1 \{ \frac{n_2}{2} \} \otimes \pi_2 \{-\frac{n_1}{2} \},$$
o˘ pour toute suite $\underline r=(0< r_1 < r_2 < \cdots < r_k=d)$, on note $P_{\underline r}$ 
le sous-groupe parabolique de $GL_d$ standard associÈ au sous-groupe de Levi 
$$GL_{r_1} \times GL_{r_2-r_1} \times \cdots \times GL_{r_k-r_{k-1}}.$$ 
\end{notas}

Rappelons qu'une reprÈsentation irrÈductible est dite supercuspidale si elle n'est pas
un sous-quotient d'une induite parabolique propre.

\begin{defi} \label{defi-rep} (cf. \cite{zelevinski2} \S 9 et \cite{boyer-compositio} \S 1.4)
Soient $g$ un diviseur de $d=sg$ et $\pi$ une reprÈsentation cuspidale irrÈductible de $GL_g(K)$. L'induite
$$\pi\{ \frac{1-s}{2} \} \times \pi \{ \frac{3-s}{2} \} \times \cdots \times \pi \{ \frac{s-1}{2} \}$$ 
possËde un unique quotient (resp. sous-espace) irrÈductible notÈ habituellement $\st_s(\pi)$ (resp.
$\speh_s(\pi)$); c'est une reprÈsentation de Steinberg (resp. de Speh) gÈnÈralisÈe.
\end{defi}

\rem du point de vue galoisien, via la correspondance de Langlands locale, la reprÈsentation
$\speh_s(\pi)$ correspond ‡ la somme directe $\sigma(\frac{1-s}{2}) \oplus \cdots \oplus
\sigma(\frac{s-1}{2})$ o˘ $\sigma$ correspond ‡ $\pi$. Plus gÈnÈralement pour
$\pi$ une reprÈsentation irrÈductible quelconque de $GL_g(K)$ associÈe ‡ $\sigma$
par la correspondance de Langlands locale, on notera $\speh_s(\pi)$ la reprÈsentation
de $GL_{sg}(K)$ associÈe, par la correspondance de Langlands locale, ‡ 
$\sigma(\frac{1-s}{2}) \oplus \cdots \oplus \sigma(\frac{s-1}{2})$.

\begin{defi} Une $\bar \Qm_l$-reprÈsentation lisse de longueur finie $\pi$ de $GL_d(K)$ est dite \emph{entiËre} s'il existe une extension finie $E/ \Qm_l$ contenue dans
$\bar \Qm_l$, d'anneau des entiers 
$\OC_E$ et une $\OC_E$-reprÈsentation $L$ de $GL_d(K)$, qui est un 
$\OC_E$-module libre, telle que $\bar \Qm_l \otimes_{\OC_E} L \simeq \pi$
et tel que $L$ est un $\OC_E GL_n(K)$-module de type fini. 
Soit $\kappa_E$ le corps rÈsiduel de $\OC_E$, on dit que 
$$\bar \Fm_l \otimes_{\kappa_E} \kappa_E \otimes_{\OC_E} L$$ 
est la rÈduction modulo $l$ de $L$. 
\end{defi}

\rem le \textit{principe de Brauer-Nesbitt} affirme que 
la semi-simplifiÈe de $\bar \Fm_l \otimes_{\OC_E} L$ est une $\bar \Fm_l$-reprÈsentation de $GL_d(K)$ 
de longueur finie qui ne dÈpend pas du choix de $L$. Son image dans le groupe de
Grothendieck sera notÈe $r_l(\pi)$ et dite \textit{la rÈduction modulo $l$ de $\pi$}.

\noindent \textit{Exemples}:  d'aprËs  \cite{vigneras-induced} V.9.2 ou \cite{dat-jl} \S 2.2.3,
la rÈduction modulo $l$ de $\speh_s(\pi)$ est irrÈductible.

\begin{defi} 
Pour $\varrho$ une $\overline \Fm_l$-reprÈsentation irrÈductible supercuspidale $\varrho$
de $GL_{g_{-1}(\varrho)}(F_v)$, on note
$$\epsilon(\varrho)=\sharp \bigl \{ \varrho \{ i \}:~ i \in \Zm \bigr \}$$ 
le cardinal de la droite de Zelevinsky de $\varrho$, et
$$m(\varrho)=\left \{ \begin{array}{ll} \epsilon(\varrho) & \hbox{si } \epsilon(\varrho)>1, \\ l & \hbox{sinon.}
\end{array} \right.$$
\end{defi}

\rem $\epsilon(\varrho)$ est un diviseur de l'ordre $e_l(q)$ de $q$ modulo $l$.

\begin{defi} (cf. \cite{vigneras-induced} III.5.14)
Pour tout $s$ de la forme
$$s=1,m(\varrho),m(\varrho)l,\cdots,m(\varrho) l^u,\cdots,$$
l'induite parabolique
$$\varrho\{\frac{1-s}{2} \} \times \cdots \times \varrho \{ \frac{s-1}{2} \}$$ 
admet un unique sous-quotient cuspidal que l'on note respectivement 
$$\rho_{-1} \simeq \varrho, ~ \rho_0, ~  \rho_1, ~ \cdots, ~  \rho_u, ~\cdots$$
\end{defi}

\rem toute $\overline \Fm_l$-reprÈsentation irrÈductible cuspidale (resp. supercuspidale) est de la forme
$\rho_u$ (resp. $\varrho_{-1}$) pour $u \geq -1$ et $\varrho$ une reprÈsentation irrÈductible
supercuspidale.

\begin{defi} \label{defi-varrho-type}
La rÈduction modulo $l$ d'une $\overline \Qm_l$-reprÈsentation irrÈductible cuspidale
$\pi_v$ Ètant irrÈductible cuspidale et donc de la forme $\rho_u$, on dira qu'elle de $\varrho$-type $u$.
Pour $u \geq 0$, on notera 
$$g_u(\varrho):=g_{-1}(\varrho)m(\varrho)l^u,$$
et $\cusp(\varrho,u)$ l'ensemble des classes d'Èquivalence des reprÈsentations irrÈductibles
cuspidales de $\varrho$-type $u$.
\end{defi}

\rem lorsque $\varrho$ sera clairement fixÈ, on notera plus simplement $g_u$ pour $g_u(\varrho)$.

\begin{notas} On note $D_{K,d}$ l'algËbre ‡ division centrale sur $K$ d'invariant $1/d$ et d'ordre
maximal $\DC_{K,d}$: la norme rÈduite $\nrd$ composÈe avec la valuation donne une identification
$D_{K,d}/\DC_{K,d} \simeq \Zm$. Pour $\pi$ une $\overline \Qm_l$-reprÈsentation irrÈductible 
supercuspidale de $GL_g(K)$ avec $d=sg$, 
$$\pi[s]_D$$
dÈsignera la reprÈsentation de $D_{K,d}^\times$ associÈe ‡ $\st_t(\pi^\vee)$ par la correspondance
de Jacquet-Langlands.
\end{notas}

\rem on rappelle qu'avec ces notations, toute reprÈsentation irrÈductible admissible de $D_{K,d}^\times$
est de la forme $\pi[s]_D$ pour $g$ dÈcrivant les diviseurs de $d$ et $\pi$ les reprÈsentations
irrÈductibles cuspidales de $GL_g(K)$.

Soit $\pi$ une $\overline \Qm_l$-reprÈsentation irrÈductible supercuspidale entiËre de $\varrho$-type $u$.
On note alors $\iota$ l'image de $\speh_{s}(\varrho^\vee)$ par la correspondance de Jacquet-Langlands 
modulaire dÈfinie par J.-F. Dat au \S 1.2.4 de \cite{dat-jl}. Autrement dit si $\pi_\varrho$ est 
un relËvement cuspidal de $\varrho$, alors 
$$\iota=r_l \bigl ( \pi_\varrho[s]_D \bigr ).$$

\begin{prop} \phantomsection \label{prop-red-D} (cf. \cite{dat-jl} proposition 2.3.3)
Avec les notations prÈcÈdentes, la rÈduction modulo $l$ de $\pi[s]_D$ est de la forme
$$\iota \{-\frac{m(\tau)-1}{2} \} \oplus \iota \{-\frac{m(\tau)-3}{2} \} \oplus \cdots \oplus \iota \{ \frac{m(\tau)-1}{2} \}$$
o˘ $\iota \{ n \}$ dÈsigne $\iota \otimes q^{-n \val \circ \nrd}$ et $m(\tau):=m(\varrho)l^u$.
\end{prop}

\subsection{GÈomÈtrie des variÈtÈs de Shimura unitaires simples}

Soit $F=F^+ E$ un corps CM avec $E/\Qm$ quadratique imaginaire, dont on fixe 
un plongement rÈel $\tau:F^+ \hookrightarrow \Rm$. Pour $v$ une place de $F$, on notera 
\begin{itemize}
\item $F_v$ le complÈtÈ du localisÈ de $F$ en $v$,

\item $\OC_v$ l'anneau des entiers de $F_v$,

\item $\varpi_v$ une uniformisante et

\item $q_v$ le cardinal du corps rÈsiduel $\kappa(v)=\OC_v/(\varpi_v)$.
\end{itemize}

Soit $B$ une algËbre ‡ 
division centrale sur $F$ de dimension $d^2$ telle qu'en toute place $x$ de $F$,
$B_x$ est soit dÈcomposÈe soit une algËbre ‡ division et on suppose $B$ 
munie d'une involution de
seconde espËce $*$ telle que $*_{|F}$ est la conjugaison complexe $c$. Pour
$\beta \in B^{*=-1}$, on note $\sharp_\beta$ l'involution $x \mapsto x^{\sharp_\beta}=\beta x^*
\beta^{-1}$ et $G/\Qm$ le groupe de similitudes, notÈ $G_\tau$ dans \cite{h-t}, dÈfini
pour toute $\Qm$-algËbre $R$ par 
$$
G(R)  \simeq   \{ (\lambda,g) \in R^\times \times (B^{op} \otimes_\Qm R)^\times  \hbox{ tel que } 
gg^{\sharp_\beta}=\lambda \}
$$
avec $B^{op}=B \otimes_{F,c} F$. 
Si $x$ est une place de $\Qm$ dÈcomposÈe $x=yy^c$ dans $E$ alors 
\addtocounter{smfthm}{1}
\begin{equation} \label{eq-facteur-v}
G(\Qm_x) \simeq (B_y^{op})^\times \times \Qm_x^\times \simeq \Qm_x^\times \times
\prod_{z_i} (B_{z_i}^{op})^\times,
\end{equation}
o˘, en identifiant les places de $F^+$ au dessus de $x$ avec les places de $F$ au dessus de $y$,
$x=\prod_i z_i$ dans $F^+$.

Dans \cite{h-t}, les auteurs justifient l'existence d'un $G$ comme ci-dessus tel qu'en outre:
\begin{itemize}
\item si $x$ est une place de $\Qm$ qui n'est pas dÈcomposÈe dans $E$ alors
$G(\Qm_x)$ est quasi-dÈployÈ;

\item les invariants de $G(\Rm)$ sont $(1,d-1)$ pour le plongement $\tau$ et $(0,d)$ pour les
autres. 
\end{itemize}

Rappelons, cf. \cite{h-t} bas de la page 90, qu'un sous-groupe ouvert compact de $G(\Am^\oo)$
est dit \og assez petit \fg{} s'il existe une place $x$ pour laquelle la projection de $U^v$ 
sur $G(\Qm_x)$ ne contienne aucun ÈlÈment d'ordre fini autre que l'identitÈ.

\begin{nota}
Soit $\IC$ l'ensemble des sous-groupes compacts ouverts \og assez petits \fg{} de $G(\Am^\oo)$.
Pour $I \in \IC$, on note $X_{I,\eta} \longrightarrow \spec F$ la variÈtÈ de Shimura associÈe, dit
de Kottwitz-Harris-Taylor.
\end{nota}

\rem pour tout $v \in \Spl$, la variÈtÈ $X_{I,\eta}$ admet un modËle projectif $X_{I,v}$ sur $\spec \OC_v$
de fibre spÈciale $X_{I,s_v}$. Pour $I$ dÈcrivant $\IC$, le systËme projectif $(X_{I,v})_{I\in \IC}$ 
est naturellement muni d'une action de $G(\Am^\oo) \times \Zm$  telle que l'action d'un ÈlÈment
$w_v$ du groupe de Weil $W_v$ de $F_v$ est donnÈe par celle de $-\deg (w_v) \in \Zm$,
o˘ $\deg=\val \circ \art^{-1}$ o˘ $\art^{-1}:W_v^{ab} \simeq F_v^\times$ est
l'isomorphisme d'Artin qui envoie les Frobenius gÈomÈtriques sur les uniformisantes.

\begin{notas} (cf. \cite{boyer-invent2} \S 1.3)
Pour $I \in \IC$, la fibre spÈciale gÈomÈtrique $X_{I,\bar s_v}$ admet une stratification de Newton
$$X_{I,\bar s_v}=:X^{\geq 1}_{I,\bar s_v} \supset X^{\geq 2}_{I,\bar s_v} \supset \cdots \supset 
X^{\geq d}_{I,\bar s_v}$$
o˘ $X^{=h}_{I,\bar s_v}:=X^{\geq h}_{I,\bar s_v} - X^{\geq h+1}_{I,\bar s_v}$ est un schÈma 
affine\footnote{cf. par exemple \cite{ito2}}, lisse de pure dimension $d-h$ formÈ des points gÈomÈtriques 
dont la partie connexe du groupe de Barsotti-Tate est de rang $h$.
Pour tout $1 \leq h <d$, nous utiliserons les notations suivantes:
$$i_{h+1}:X^{\geq h+1}_{I,\bar s_v} \hookrightarrow X^{\geq h}_{I,\bar s_v}, \quad
j^{\geq h}: X^{=h}_{I,\bar s_v} \hookrightarrow X^{\geq h}_{I,\bar s_v},$$
ainsi que $j^{=h}=i_h \circ j^{\geq h}$.
\end{notas}

Rappelons que les strates de Newton non supersinguliËres sont gÈomÈtriquement induites au sens
suivant, cf. \cite{boyer-invent}: il existe un sous-schÈma fermÈ $X^{=h}_{I,\bar s_v,\overline{1_h}}$
de $X^{=h}_{I,\bar s_v}$ stable par les correspondances de Hecke associÈes au parabolique
$P_{h,d}(F_v)$ et tel que si la composante en $v$ de $I$ est $\ker(GL_d(\OC_v) \twoheadrightarrow
GL_d(\OC_v/\varpi_v^n))$ alors
$$X^{=h}_{I,\bar s_v}=\coprod_{g \in GL_d(\OC_v/\varpi_v^n)/P_{h,d}(\OC_v/\varpi_v^n)} 
g .X^{=h}_{I,\bar s_v,\overline{1_h}}.$$

\begin{nota}
Pour $g  \in GL_d(\OC_v/\varpi_v^n)/P_{h,d}(\OC_v/\varpi_v^n)$, on notera 
$V_g \subset (\OC_v/\varpi_v^n)^d$, l'image par $g$ du sous-espace $V_{\overline{1_h}}$
engendrÈ par les $h$-premiers vecteurs de la base canonique. Pour $V \subset (\OC_v/\varpi_v^n)^d$
de la forme $g. V_{\overline{1_h}}$, on notera 
$$X^{\geq h}_{I,\bar s_v,V}:=g. X^{\geq h}_{I,\bar s_v,\overline{1_h}}.$$
\end{nota}

\begin{nota} On note
$X^{\geq h}_{I,\bar s_v,\overline{1_h}}$ l'adhÈrence schÈmatique de $X^{=h}_{I,\bar s_v,\overline{1_h}}$,
et pour tout $h' \geq h$, 
$$X^{\geq h'}_{I,\bar s_v,\overline{1_h}} := X^{\geq h}_{I,\bar s_v,\overline{1_h}}
\cap X^{\geq h'}_{I,\bar s_v}.$$
On utiliser $j^{\geq h'}_{\overline{1_h}}: X^{=h'}_{\IC,\bar s_v,\overline{1_h}} \hookrightarrow
X^{\geq h'}_{\IC,\bar s_v,\overline{1_h}}$ ainsi que $j^{=h'}_{\overline{1_h}}:=i^{h'}_{\overline{1_h}}
\circ j^{\geq h'}_{\overline{1_h}}$ o˘ $i^{h'}_{\overline{1_h}}:X^{\geq h'}_{\IC,\bar s_v,\overline{1_h}}
\hookrightarrow X^{\geq 1}_{\IC,\bar s_v}$.
\end{nota}

%
%

\begin{nota} \label{nota-spl}
On fixe un nombre premier $l$ non ramifiÈ dans $E$ et
on note $\Spl$ l'ensemble des places $v$ de $F$ telles que $p_v:=v_{|\Qm} \neq l$
est dÈcomposÈ dans $E$ et $B_v^\times \simeq GL_d(F_v)$. Pour $I \in \IC$, on notera
$\Spl(I)$ le sous-ensemble de $\Spl$ des places ne divisant pas le niveau $I$.
\end{nota}

\subsection{Faisceaux pervers d'Harris-Taylor}

¿ toute reprÈsentation irrÈductible
admissible $\tau_v$ de $D_{v,h}^\times$, Harris et Taylor associent 
un systËme local de Hecke $\FC_{\tau_v,\IC,\overline{1_h}}$ sur $X_{\IC,\bar s_v,\overline{1_h}}^{=h}$,
au sens du \S 1.4.5 de \cite{boyer-invent2}, avec une action de
$G(\Am^{\oo,p}) \times \Qm_p^\times \times P_{h,d}(F_v) \times 
\prod_{i=2}^r (B_{v_i}^{op})^\times  \times \Zm$
qui d'aprËs \cite{h-t} p.136, se factorise par $G^{(h)}(\Am^\oo)/\DC_{F_v,h}^\times$ via
\begin{equation} \label{eq-action-tordue}
(g^{\oo,p},g_{p,0},c,g_v,g_{v_i},k) \mapsto (g^{p,\oo},g_{p,0}q^{k-v (\det g_v^c)}, 
g_v^{et},g_{v_i}, \delta).
\end{equation}
o˘ $G^{(h)}(\Am^\oo):=G(\Am^{\oo,p}) \times \Qm_p^\times \times GL_{d-h}(F_v) \times
\prod_{i=2}^r (B_{v_i}^{op})^\times \times D_{F_v,h}^\times$,
$g_v=\left ( \begin{array}{cc} g_v^c & * \\ 0 & g_v^{et} \end{array} \right )$ et
$\delta \in D_{v,h}^\times$ est tel que $v(\rn (\delta))=k+v(\det g_v^c)$.
On note $\FC_{\tau_v,\IC}$ le faisceau sur $X_{\IC,\bar s_v}^{=h}$ induit associÈ:
$$\FC_{\tau_v,\IC}:= \FC_{\tau_v,\IC,\overline{1_h}} \times_{P_{h,d}(F_v)} GL_d(F_v).$$

Pour $\pi_v$ une reprÈsentation irrÈductible cuspidale de $GL_g(F_v)$ et $t$ un entier strictement positif tel que $tg \leq d$,
on introduit suivant \cite{boyer-invent2} la notation $\FC(\pi_v,t)_{\overline{1_{tg}}}$ 
(resp. $\FC(\pi_v,t)$) pour dÈsigner 
le faisceau de Hecke sur $X_{\IC,\bar s_v,\overline{1_{tg}}}^{=tg}$ (resp. $X_{\IC,\bar s_v}^{=tg}$)
prÈcÈdemment notÈ $\FC_{\pi_v[t]_D,\IC,\overline{1_{tg}}}$ (resp. $\FC_{\pi_v[t]_D,\IC}$).

\begin{nota}
Avec les notations prÈcÈdentes, $HT_{\overline \Qm_l,\overline{1_{tg}}}(\pi_v,\Pi_t)$ dÈsignera le 
$W_v$-faisceau pervers de Hecke sur $X_{\IC,\bar s_v,\overline{1_{tg}}}^{=tg}$ dÈfini par
$$HT_{\overline \Qm_l,\overline{1_{tg}}}(\pi_v,\Pi_t):=\FC(\pi_v,t)_{\overline{1_{tg}}}[d-tg] \otimes 
\Xi^{\frac{tg-d}{2}} \otimes \Pi_t,$$
o˘ $\Pi_t$ (resp. $\Xi$) une reprÈsentation de $GL_{tg}(F_v)$ (resp. de $\Zm$). En ce qui concerne
les actions:
\begin{itemize}
\item celle de $G(\Am^{\oo,v})$ est donnÈe par son action naturelle sur $\FC(\pi_v,t)_{\overline{1_{tg}}}$,

\item celle de $P_{h,d}(F_v)$, ‡ travers son Levi $GL_h(F_v) \times GL_{d-h}(F_v)$,
est donnÈe par l'action naturelle de $GL_{d-h}(F_v)$ sur $\FC(\pi_v,t)_{\overline{1_{tg}}}$ 
et l'action diagonale
de $\Zm$ sur $\FC(\pi_v,t)_1$ et $\Xi^{\frac{tg-d}{2}}$, tandis que celle de $GL_h(F_v)$
est donnÈe par son action diagonale sur $\Pi_t \otimes \Xi^{\frac{tg-d}{2}}$ o˘, sur le deuxiËme facteur,
on utilise la valuation du dÈterminant.
\end{itemize}
Pour tout $h \leq tg$, on notera 
$$HT_{\overline \Qm_l,\overline{1_h}}(\pi_v,\Pi_t) \quad \hbox{resp. } 
HT_{\overline \Qm_l}(\pi_v,\Pi_t)$$ 
pour dÈsigner la version induite ‡ toute la strate de Newton $X_{\IC,\bar s_v,\overline{1_h}}^{=tg}$
(resp. ‡ $X_{\IC,\bar s_c}^{=tg}$).
Enfin on otera l'indice $\overline \Qm_l$, pour dÈsigner un $\overline \Zm_l$-rÈseau stable 
sous l'action des correspondances de Hecke associÈes ‡ $P_{h,d}(F_v)$.
\end{nota}

\begin{defi}
Le faisceau pervers d'Harris-Taylor associÈ ‡ $\st_t(\pi_v)$ est par dÈfinition l'extension intermÈdiaire
$$P_{\overline \Qm_l}(\pi_v,t):=\lexp p j^{=tg}_{!*} HT_{\overline \Qm_l}(\pi_v,\st_t(\pi_v)).$$
\end{defi}

\rem sur $\overline \Zm_l$, on deux notions d'extensions intermÈdiaires $\lexp p j^{=tg}_{!*}$
et $\lexp {p+} j^{=tg}_{!*}$.

\begin{nota}
Suivant \cite{h-t}, ‡ toute $\Cm$-reprÈsentation irrÈductible algÈbrique de dimension finie $\xi$ 
de $G$, on associe un systËme local $V_\xi$ sur $X_\IC$ et on dÈcore tous nos faisceaux 
d'un indice $\xi$ pour dÈsigner leur torsion par $V_\xi$:
par exemple $HT_\xi(\pi_v,\Pi_t):=HT(\pi_v,\Pi_t) \otimes V_\xi$.
\end{nota}

\section{Torsion dans la cohomologie}

\renewcommand{\theequation}{\arabic{section}.\arabic{subsection}.\arabic{smfthm}}
\renewcommand{\thesmfthm}{\arabic{section}.\arabic{subsection}.\arabic{smfthm}}

On fixe ‡ prÈsent une $\overline \Fm_l$-reprÈsentation irrÈductible supercuspidale $\varrho$ de $GL_{g_{-1}}(F_v)$ ainsi qu'une $\overline \Qm_l$-reprÈsentation irrÈductible cuspidale entiËre 
$\pi_{v,-1}$ de $\varrho$-type $-1$.
Contrairement au paragraphe suivant, on considËre ici des niveaux $I$ ramifiÈs en $v$.

\subsection{Rappels sur la $\overline \Qm_l$-cohomologie des systËmes locaux d'Harris-Taylor}
\label{para-rappel-coho}

Pour $1 \leq h=tg_{-1} \leq d$, on note $\IC_v(h)$ l'ensemble des sous-groupes compacts ouverts de la forme 
$$U_v(\underline m,h):= 
U_v(\underline m^v) \times \left ( \begin{array}{cc} I_h & 0 \\ 0 & K_v(m_1) \end{array} \right ),$$
o˘ $K_v(m_1)=\ker \bigl ( GL_{d-h}(\OC_v) \longrightarrow GL_{d-h}(\OC_v/ (\varpi_v^{m_1})) \bigr )$.

\begin{nota} On note $[H^i_{\xi,!}(tg_{-1},\pi_{v-1})]$ l'image de 
$$\lim_{\atop{\longrightarrow}{I \in \IC_v(h)}} H^i(X_{I,\bar s_v,\overline{1_h}}^{\geq h}, 
j^{\geq tg_{-1}}_{\overline{1_h},!} HT_{\overline \Qm_l,\overline{1_h},\xi}(\pi_{v,-1},\Pi_t)[d-h]) 
$$ 
dans le groupe de Grothendieck $\groth(v,h)$ des reprÈsentations admissibles de 
$G(\Am^{\oo}) \times GL_{d-h}(F_v) \times GL_h(F_v) \times \Zm$.
\end{nota}

\rem l'action de $\sigma \in W_v$ est donnÈe par celle de $-\deg \sigma \in \Zm$, 
composÈe avec celle de $\art^{-1} (\sigma)$ sur le facteur $\Qm_p^\times$ de $G(\Qm_p)$.

\begin{defi}
On dit d'une reprÈsentation irrÈductible automorphe $\Pi$ qu'elle est $\xi$-cohomologique
s'il existe un entier $i$ tel que
$$H^i \bigl ( ( \lie ~G(\Rm)) \otimes_\Rm \Cm,U,\Pi_\oo \otimes \xi^\vee \bigr ) \neq (0),$$
o˘ $U$ est un sous-groupe compact maximal modulo le centre de $G(\Rm)$.
\end{defi}

\begin{defi} Pour $\pi_v$ une reprÈsentation irrÈductible cuspidale de $GL_g(F_v)$, 
$$s_\xi(\pi_v)$$ 
dÈsignera le plus grand entier $s$ tel qu'il existe une reprÈsentation automorphe $\xi$-cohomologique
$\Pi$ telle que sa composante locale en $v$ est de la forme $\speh_{s}(\pi'_{v}) \times ?$ o˘
$\pi'_v$ est inertiellement Èquivalente ‡ $\pi_v$ et $?$ dÈsigne 
une reprÈsentation de $GL_{d-sg_{-1}}(F_v)$ que l'on ne cherche pas ‡ prÈciser.
\end{defi}

\begin{prop} (cf. \cite{boyer-compositio} \S 5 ou \cite{boyer-imj} \S 3.3) \label{prop-torsion-Ql}
Pour tout $1 \leq t \leq s_{-1}:=\lfloor \frac{d}{g_{-1}} \rfloor$, et pour tout $i$ tel que 
$|i| > s_\xi(\pi_{v,-1})-t$, alors $[H^i_{\xi,!}(tg_{-1},\pi_{v,-1})]$ est nul. Pour 
$t \leq s_\xi(\pi_{v,-1})$ et pour $i=s_\xi(\pi_{v,-1})-t$), alors $[H^i_{\xi,!}(tg_{-1},\pi_{v,-1})]$
est non nul.
\end{prop}

\subsection{Torsion dans la cohomologie des faisceaux pervers d'Harris-Taylor}
\label{para-torsion-coho}

Dans cette section nous allons rappeler comment, d'aprËs le \S 4.5 de \cite{boyer-aif}, on peut
construire de la torsion dans la cohomologie d'un faisceau pervers d'Harris-Taylor. 

\begin{nota} On note $\Fm(-)$ le foncteur $(-) \otimes^\Lm_{\overline \Zm_l} \overline \Fm_l$.
\end{nota} 

\begin{prop} \label{prop-pp} (cf. \cite{boyer-entier})
Pour tout $1 \leq t \leq s_{-1}:=\lfloor \frac{d}{g_{-1}} \rfloor$, on a
$$\lexp p j^{=tg}_{!*} HT_{\overline \Zm_l,\xi}(\pi_{v,-1},\Pi_t) \simeq \lexp {p+} j^{=tg}_{!*} 
HT_{\overline \Zm_l,\xi}(\pi_{v,-1},\Pi_t).$$
\end{prop}

On note $HT_{\overline \Fm_l,\xi}(\pi_{v,-1},\Pi_t)$, la rÈduction modulo $l$ d'un  
$\overline \Zm_l$-rÈseau stable de $HT_{\overline \Zm_l,\xi}(\pi_{v,-1},\Pi_t)$.

\begin{coro} \label{coro-torsion1}
Pour $t \leq s_\xi(\pi_{v,-1})$,
$$H^i(X_{\IC,\bar s_v,\overline{1_h}}^{\geq h}, j^{\geq tg_{-1}}_{\overline{1_h},!}  
HT_{\overline \Fm_l,\xi}(\pi_{v,-1},\Pi_t))$$
est nul pour tout $i>s_{\xi}(\pi_{v,-1})-t$ et non nul pour $i=s_{\xi}(\pi_{v,-1})-t$.
Pour  $t > s_\xi(\pi_{v,-1})$, tous ces groupes de cohomologie sont nuls.
\end{coro}

\rem autrement dit la conclusion de la proposition \ref{prop-torsion-Ql} est encore valable 
sur $\overline \Fm_l$.

\begin{proof}
Rappelons l'ÈgalitÈ suivante sur $\overline \Qm_l$ et dans le groupe de Grothendieck des
faisceaux pervers Èquivariant de Hecke, cf. \cite{boyer-invent2} corollaire 5.4.1
\addtocounter{smfthm}{1}
\begin{multline} \label{egalite-hij}
j^{= tg_{-1}}_{\overline{1_h},!} HT_{\overline \Qm_l,\xi,\overline{1_h}}(\pi_{v,-1},\Pi_t)= 
\lexp p j^{= tg_{-1}}_{\overline{1_h},!*} HT_{\overline \Qm_l,\xi,\overline{1_h}}(\pi_{v,-1},\Pi_t) + \\
\sum_{i=1}^{s_{-1}-t} \lexp p j^{=(t+i)g_{-1}}_{\overline{1_h},!*} HT_{\overline \Qm_l,\xi,\overline{1_h}}
(\pi_{v,-1},\Pi_t \{ i \frac{g-1}{2} \} \times \st_i(\pi_{v,-1} \{-t \frac{g-1}{2}\} )) \otimes \Xi^{\frac{i}{2}}
\end{multline}
La filtration de stratification construite dans \cite{boyer-torsion} 
$$0=\Fil^{t-s_{-1}-1}_!(\pi_{v,-1},\Pi_t) \subset \Fil^{t-s}_!(\pi_{v,-1},\Pi_t) \subset \cdots 
\Fil^{0}_!(\pi_{v,-1},\Pi_t)$$
a pour graduÈs, d'aprËs \cite{boyer-duke}, les
$$\gr^i_!(\pi_{v,-1},\Pi_t) \simeq \lexp p  j^{=(t+i)g_{-1}}_{\overline{1_h},!*} 
HT_{\overline \Zm_l,\xi,\overline{1_h}}(\pi_{v,-1},\Pi_t \{ i \frac{g-1}{2} \} 
\times \st_i(\pi_{v,-1} \{-t \frac{g-1}{2}\} )) \otimes \Xi^{\frac{i}{2}},$$
pour certains rÈseaux stables qu'il est inutile de prÈciser.
%
On considËre alors la suite spectrale associÈe ‡ cette filtration et calculant les groupes de cohomologie
de $j^{= tg_{-1}}_{\overline{1_h},!} HT_{\overline \Zm_l,\xi,\overline{1_h}}(\pi_{v,-1},\Pi_t)$
\addtocounter{smfthm}{1}
\begin{equation} \label{eq-ss-filt1}
E_1^{p,q}=H^{p+q}(X_{\IC,\bar s_v}, \gr^{-p}_!(\pi_{v,-1},\Pi_t)) \Rightarrow H^{p+q}(X_{\IC,\bar s_v},
j^{= tg_{-1}}_{\overline{1_h},!} HT_{\overline \Zm_l,\xi,\overline{1_h}}(\pi_{v,-1},\Pi_t)).
\end{equation}

\begin{lemm} \label{lem-coho1}
Pour tout 
$i < t-s_{\xi}(\pi_{v,-1})$ (resp. $i=t-s_{\xi}(\pi_{v,-1})$) le groupe de cohomologie 
$H^i(X_{\IC,\bar s_v},\lexp p j^{=tg_{-1}}_{\overline{1_h},!*} HT_{\overline \Zm_l,\xi,\overline{1_h}}
(\pi_{v,-1},\Pi_t))$ est nul (resp. non nul et sans torsion).
\end{lemm}

\begin{proof}
On raisonne par rÈcurrence sur $t$ de $s_{-1}$ ‡ $1$. Pour $t=s_{-1}$, on rappelle que
\begin{multline*}
\lexp p j^{=s_{-1}g_{-1}}_{\overline{1_h},!*} HT_{\overline \Zm_l,\xi,\overline{1_h}}
(\pi_{v,-1},\Pi_{s_{-1}}) \simeq  j^{=s_{-1}g_{-1}}_{\overline{1_h},!} HT_{\overline \Zm_l,\xi,\overline{1_h}}
(\pi_{v,-1},\Pi_{s_{-1}}) \\
\simeq  j^{=s_{-1}g_{-1}}_{\overline{1_h},*} HT_{\overline \Zm_l,\xi,\overline{1_h}}(\pi_{v,-1},\Pi_{s_{-1}}) 
\simeq \lexp {p+}  j^{=s_{-1}g_{-1}}_{\overline{1_h},!*} HT_{\overline \Zm_l,\xi,\overline{1_h}}
(\pi_{v,-1},\Pi_{s_{-1}})
\end{multline*}
de sorte que, considÈrant que $X^{=h}_{I,\bar s_v,\overline{1_h}}$ est affine, les
$H^i(X_{\IC,\bar s_v,\overline{1_h}}, j^{=s_{-1}g_{-1}}_{\overline{1_h},!} 
HT_{\overline \Zm_l,\xi,\overline{1_h}} (\pi_{v,-1},\Pi_{s_{-1}})$ (resp.
$H^i(X_{\IC,\bar s_v,\overline{1_h}}, j^{=s_{-1}g_{-1}}_{\overline{1_h},*} 
HT_{\overline \Zm_l,\xi,\overline{1_h}}(\pi_{v,-1},\Pi_{s_{-1}})$) son nuls pour $i<0$ 
(resp. $i>0$), et donc concentrÈs en degrÈ mÈdian.
En utilisant en outre que $j^{\geq h}_{\overline{1_h}}$ est affine, de sorte que $j^{=h}_{\overline{1_h},!}$ et 
$j^{=h}_{\overline{1_h},*}$ commutent avec le 
foncteur de rÈduction modulo $l$, on en dÈduit qu'il en est de mÍme pour
$H^i(X_{\IC,\bar s_v,\overline{1_h}}, j^{=s_{-1}g_{-1}}_{\overline{1_h},!} 
HT_{\overline \Fm_l,\xi,\overline{1_h}}(\pi_{v,-1},\Pi_{s_{-1}})$ ce qui impose que
$H^0(X_{\IC,\bar s_v,\overline{1_h}}, j^{=s_{-1}g_{-1}}_{\overline{1_h},!} 
HT_{\overline \Zm_l,\xi,\overline{1_h}}(\pi_{v,-1},\Pi_{s_{-1}})$ est sans torsion, cf. la suite
exacte courte gÈnÈrale (\ref{eq-sec-torsion0}) rappelÈe plus bas.
Le rÈsultat dÈcoule alors de la proposition \ref{prop-torsion-Ql}.

Supposons alors le rÈsultat acquis jusqu'au rang $t+1$ et revenons ‡ l'Ètude de la suite spectrale
(\ref{eq-ss-filt1}). 
\begin{itemize}
\item ConsidÈrons tout d'abord le cas o˘ $t> s_{\xi}(\pi_{v,-1})$ auquel cas, par
rÈcurrence tous les $E_1^{p,q}$ avec $p>0$ sont nuls. Or $X^{=tg_{-1}}_{I,\bar s_v,\overline{1_h}}$
Ètant affine, les $E_\oo^{p+q}$ sont nuls pour $p+q<0$, ce qui impose que les 
$H^i(X_{\IC,\bar s_v,\overline{1_h}},\lexp p j^{=tg_{-1}}_{\overline{1_h},!*} 
HT_{\overline \Zm_l,\xi,\overline{1_h}}(\pi_{v,-1},\Pi_t))$ sont nuls pour 
$i<0$ et sans torsion pour $i=0$. Or d'aprËs la proposition \ref{prop-torsion-Ql}, on en dÈduit
qu'il est nul aussi pour $i=0$ et par dualitÈ, d'aprËs la proposition \ref{prop-pp}, nul pour tout $i$.

\item Supposons ‡ prÈsent $t \leq s_{\xi}(\pi_{v,-1})$.
D'aprËs l'hypothËse de rÈcurrence, pour $p>0$ fixÈ, les $E_1^{p,q}$ sont nuls pour
$p+q<t+p-s_{\xi}(\pi_{v,-1})$ et sans torsion pour $p+q=t+p-s_{\xi}(\pi_{v,-1})$. Par ailleurs comme
$X^{=tg_{-1}}_{I,\bar s_v,\overline{1_h}}$ est affine, les $E_\oo^{p+q}$ sont nuls pour $p+q<0$, 
de sorte que les
$E_1^{0,q}$ sont nuls pour $q<t-s_{\xi}(\pi_{v,-1})$ et $E_1^{0,t-s_{\xi}(\pi_{v,-1})}$ est sans torsion,
d'o˘ le rÈsultat. 
\end{itemize}
\end{proof}

Ainsi en utilisant que, d'aprËs la proposition \ref{prop-pp}, les $\gr_!^{-p}(\pi_v,\Pi_t)$ sont autoduaux 
en tant que faisceaux pervers, on en dÈduit que, pour $1 \leq p \leq s_{-1}-t$ fixÈ, les termes $E_1^{p,q}$ 
de la suite spectrale \ref{eq-ss-filt1} sont nuls pour  $p+q\geq s_{\xi}(\pi_{v,-1})-t-p$ et 
donc $E_\oo^{p+q}$ est nul pour tout $p+q > s_{\xi}(\pi_{v,-1})-t$ d'o˘ le rÈsultat.
\end{proof}

\rem l'ÈnoncÈ d'annulation pour $i > s_\xi(\pi_{v,-1})-t$ est encore valable ‡ niveau fini
et pour la non annulation pour $i=s_\xi(\pi_{v,-1})-t$, il suffit de prendre un niveau suffisamment petit
de sorte que la non annulation soit vÈrifiÈe sur $\overline \Qm_l$.

\begin{coro} 
Le nombre $s_{\xi}(\pi_{v,-1})$ ne dÈpend que de $\varrho$, on le notera alors $s_{\xi}(\varrho)$.
\end{coro}

\begin{proof}
Rappelons que pour $P$ un $\overline \Zm_l$-faisceau
pervers sans torsion, on a 
\addtocounter{smfthm}{1}
\begin{equation} \label{eq-sec-torsion0}
0 \rightarrow H^i(X,P) \otimes_{\overline \Zm_l} \overline \Fm_l \longrightarrow
H^i(X,P \otimes^\Lm_{\overline \Zm_l} \overline \Fm_l) \longrightarrow H^{i+1}(X,P)[l] \rightarrow 0,
\end{equation}
de sorte que $s_{\xi}(\pi_{v,-1})$ est le plus grand entier $i$ tel que 
$H^i(X_{\IC,\bar s_v,\overline{1_h}}, \Fm j^{=g_{-1}}_{\overline{1_h},!} 
HT_{\overline \Zm_l,\xi,\overline{1_h}} (\pi_{v,-1},\Pi_1))$ est non nul, ce qui, comme 
$\Fm j^{=g_{-1}}_!=j^{=g_{-1}}_! \Fm$, ne dÈpend donc que de la rÈduction modulo $l$ de $\pi_{v,-1}$.
\end{proof}

\begin{prop} \label{prop-equivalence0}
Pour  $1 \leq t \leq \lfloor \frac{d}{g_{-1}} \rfloor$ et $h=tg_{-1}$, la torsion de 
$$H^{s_\xi(\varrho)-t}(X_{\IC,\bar s_v,\overline{1_h}}^{\geq tg_{-1}}, 
j^{\geq tg_{-1}}_{\overline{1_h},!}  HT_{\overline \Zm_l,\xi,\overline{1_h}}(\pi_{v,-1},\Pi_t))$$
est non nulle si et seulement si celle de
$$H^1(X_{\IC,\bar s_v,\overline{1_h}}^{\geq (s_\xi(\varrho)-1)g_{-1}}, 
j^{\geq (s_\xi(\varrho)-1)g_{-1}}_{\overline{1_h},!}  HT_{\overline \Zm_l,\xi,\overline{1_h}}
(\pi_{v,-1},\Pi_{s_\xi(\varrho)-1}))$$ 
est non nulle, quelle que soit la reprÈsentation 
$\Pi_{s_\xi(\varrho)-1}$ de $P_{h,(s_\xi(\varrho)-1)g_{-1}}(F_v)$ considÈrÈe.
\end{prop}

\rem l'ÈnoncÈ prÈcÈdent est indÈpendant de la reprÈsentation infinitÈsimale $\Pi_t$, on pourrait
faire une formulation plus lÈgËre avec les $\FC(\pi_{v,-1},t)$ mais comme la dÈmonstration
fait intervenir naturellement ces parties infinitÈsimales, on a prÈfÈrÈ les laisser apparaÓtre.

\begin{proof}
Notons tout d'abord que d'aprËs ce qui prÈcËde pour $t>s_\xi(\varrho)$ (resp. $t=s_\xi(\varrho)$)
tous les groupes de cohomologie de $j^{\geq tg_{-1}}_{\overline{1_h},!}  
HT_{\overline \Zm_l,\xi,\overline{1_h}}(\pi_{v,-1},\Pi_t)$ et de
$\lexp p j^{\geq tg_{-1}}_{\overline{1_h},!*}  HT_{\overline \Zm_l,\xi,\overline{1_h}}(\pi_{v,-1},\Pi_t)$ 
sont nuls (resp. sauf pour $i=0$ auquel cas il est sans torsion). 
Il dÈcoule de la suite spectrale (\ref{eq-ss-filt1}) et du lemme \ref{lem-coho1}, que
$$H^{s_\xi(\varrho)-t}(X_{\IC,\bar s_v,\overline{1_h}}^{\geq tg_{-1}}, j^{\geq tg_{-1}}_{\overline{1_h},!}  
HT_{\overline \Zm_l,\xi,\overline{1_h}} (\pi_{v,-1},\Pi_t)) \simeq H^{s_\xi(\varrho)-t}
(X_{\IC,\bar s_v,\overline{1_h}}^{\geq tg_{-1}},
\lexp p j^{\geq tg_{-1}}_{\overline{1_h},!*}  HT_{\overline \Zm_l,\xi,\overline{1_h}} (\pi_{v,-1},\Pi_t)).$$
Ainsi par dualitÈ et d'aprËs la proposition \ref{prop-pp}, on est amenÈ ‡ Ètudier la torsion de
$H^{t-s_\xi(\varrho)+1}(X_{\IC,\bar s_v,\overline{1_h}}^{\geq tg_{-1}}, 
\lexp p j^{\geq tg_{-1}}_{\overline{1_h},!*}  HT_{\overline \Zm_l,\xi,\overline{1_h}}(\pi_{v,-1},\Pi_t))$.
Le rÈsultat dÈcoule alors trivialement du lemme suivant.
\end{proof}

\begin{lemm} \label{lem-equivalence}
Pour tout $t \geq s_\xi(\varrho)-1$, on a l'Èquivalence
\begin{multline*}
H^0_{tor}(X_{I,\bar s_v,\overline{1_h}}^{\geq (s_\xi(\varrho)-1)g_{-1}}, 
\lexp p j^{\geq (s_\xi(\varrho)-1)g_{-1}}_{\overline{1_h},!*}  
HT_{\overline \Zm_l,\xi,\overline{1_h}}(\pi_{v,-1},\Pi_{s_\xi(\varrho)-1})) \neq (0) \\
\Leftrightarrow \\
H_{tor}^{t-s_\xi(\varrho)+1}(X_{I,\bar s_v,\overline{1_h}}^{\geq tg_{-1}}, 
\lexp p j^{\geq tg_{-1}}_{\overline{1_h},!*}  HT_{\overline \Zm_l,\xi,\overline{1_h}}(\pi_{v,-1},\Pi_t)) \neq (0).
\end{multline*}
\end{lemm}

\begin{proof}
Supposons dans un premier temps que le membre de gauche de l'Èquivalence prÈcÈdente soit vÈrifiÈ.
ConsidÈrons alors la suite exacte courte
\begin{multline*}
0 \rightarrow \Fil^{-2}_!(\pi_{v,-1},\Pi_{s_\xi(\varrho)-1}) \longrightarrow 
j^{= (s_\xi(\varrho)-1)g_{-1}}_{\overline{1_h},!}  HT_{\overline \Zm_l,\xi,\overline{1_h}}
(\pi_{v,-1},\Pi_{s_\xi(\varrho)-1}) \\ \longrightarrow
\Fil^0_!(\pi_{v,-1},\Pi_{s_\xi(\varrho)-1})/\Fil^{-2}_! (\pi_{v,-1},\Pi_{s_\xi(\varrho)-1})\rightarrow 0,
\end{multline*}
de sorte que pour tout $i$, on a
\begin{multline*}
H^i(X_{I,\bar s_v,\overline{1_h}}^{\geq (s_\xi(\varrho)-1)g_{-1}}, 
j^{\geq (s_\xi(\varrho)-1)g_{-1}}_{\overline{1_h},!}  
HT_{\overline \Zm_l,\xi,\overline{1_h}}(\pi_{v,-1},\Pi_{s_\xi(\varrho)-1})) \simeq \\
H^i(X_{I,\bar s_v,\overline{1_h}}^{\geq (s_\xi(\varrho)-1)g_{-1}},\Fil^0_!(\pi_{v,-1},\Pi_{s_\xi(\varrho)-1})/
\Fil^{-2}_! (\pi_{v,-1},\Pi_{s_\xi(\varrho)-1})).
\end{multline*}
De la suite exacte courte
\begin{multline*}
0 \rightarrow \lexp p j^{=s_\xi(\varrho)g_{-1}}_{\overline{1_h},!*} HT_{\overline \Zm_l,\xi,\overline{1_h}}
(\pi_{v,-1},\Pi_{s_\xi(\varrho)} ) \otimes \Xi^{\frac{1}{2}}
\longrightarrow \\ \Fil^0_!(\pi_{v,-1},\Pi_{s_\xi(\varrho)-1})/
\Fil^{-2}_! (\pi_{v,-1},\Pi_{s_\xi(\varrho)-1}) \\ \longrightarrow 
\lexp p j^{=(s_\xi(\varrho)-1)g_{-1}}_{\overline{1_h},!*} HT_{\overline \Zm_l,\xi,\overline{1_h}}
(\pi_{v,-1},\Pi_{s_\xi(\varrho)-1}) \rightarrow 0
\end{multline*}
o˘
$$ \Pi_{s_\xi(\varrho)} := 
\ind_{P_{h,s_\xi(\varrho)-1)g_{-1},s_\xi(\varrho)g_{-1}}(F_v)}^{P_{h,s_\xi(\varrho)g_{-1}}(F_v)}
\bigl (\Pi_{s_\xi(\varrho)-1} \{ \frac{g-1}{2} \} \otimes \pi_{v,-1} \{ (1-s_\xi(\varrho))\frac{g-1}{2} \} \bigr ).$$
En utilisant
\begin{itemize}
\item d'une part que les groupes de cohomologie de $\lexp p j^{s_\xi(\varrho)g_{-1}}_{\overline{1_h},!*} 
HT_{\overline \Zm_l,\xi,\overline{1_h}} (\pi_{v,-1},\Pi_{s_\xi(\varrho)})$ sont tous nuls sauf en degrÈ $0$
auquel cas la torsion est nulle, 

\item et que d'autre part, $X^{=(s_\xi(\varrho)-1)g_{-1}}_{\IC,\bar s_v,\overline{1_h}}$ est affine
de sorte que la cohomologie de $j^{\geq (s_\xi(\varrho)-1)g_{-1}}_{\overline{1_h},!}  
HT_{\overline \Zm_l,\xi,\overline{1_h}} (\pi_{v,-1},\Pi_{s_\xi(\varrho)-1})$ et donc de 
$\Fil^0_!(\pi_{v,-1},\Pi_{s_\xi(\varrho)-1})/ \Fil^{-2}_! (\pi_{v,-1},\Pi_{s_\xi(\varrho)-1})$ est nul en degrÈ $<0$
et sans torsion en degrÈ $0$,
\end{itemize}
on en dÈduit que le morphisme 
$$H^0(X^{\geq 1}_{\IC,\bar s_v}, \lexp p j^{s_\xi(\varrho)g_{-1}}_{\overline{1_h},!*} 
HT_{\overline \Zm_l,\xi,\overline{1_h}} (\pi_{v,-1},\Pi_{s_\xi(\varrho)})) \longrightarrow
H^0(X^{\geq 1}_{\IC,\bar s_v}, \Fil^0_!(\pi_{v,-1},\Pi_{s_\xi(\varrho)-1})/\Fil^{-2}_! (\pi_{v,-1},
\Pi_{s_\xi(\varrho)-1}))$$
n'est pas strict, i.e. la torsion du conoyau est non nulle.

\rem Notons que pour $h=(s_\xi(\varrho)-1)g_{-1}$, la reprÈsentation $\Pi_{s_\xi(\varrho)}$
est irrÈductible pourvu que $\Pi_{s_\xi(\varrho)-1}$ le soit elle mÍme. Pour $h<(s_\xi(\varrho)-1)g_{-1}$,
la suite exacte courte prÈcÈdente est une somme directe indexÈe par les composantes
$X^{=(s_\xi(\varrho)-1)g_{-1}}_{\IC,\bar s_v}$ contenue dans $X^{\geq h}_{\IC,\bar s_v,\overline{1_h}}$
de sorte que pour tout $(0) \neq \widetilde{\Pi_{s_\xi(\varrho)}} \hookrightarrow \Pi_{s_\xi(\varrho)}$,
le morphisme composÈ
$$\xymatrix{
H^0(X^{\geq 1}_{\IC,\bar s_v}, \lexp p j^{s_\xi(\varrho)g_{-1}}_{\overline{1_h},!*} 
HT_{\overline \Zm_l,\xi,\overline{1_h}} (\pi_{v,-1},\widetilde{\Pi_{s_\xi(\varrho)}})) \ar[d] \ar[dr] \\
H^0(X^{\geq 1}_{\IC,\bar s_v}, \lexp p j^{s_\xi(\varrho)g_{-1}}_{\overline{1_h},!*} 
HT_{\overline \Zm_l,\xi,\overline{1_h}} (\pi_{v,-1},\Pi_{s_\xi(\varrho)})) \ar[r] &
H^0(X^{\geq 1}_{\IC,\bar s_v}, \Fil^0_!(\pi_{v,-1},\Pi_{s_\xi(\varrho)-1})/\Fil^{-2}_! (\pi_{v,-1},
\Pi_{s_\xi(\varrho)-1}))
}$$
n'est pas strict.

\medskip

ConsidÈrons ‡ prÈsent le groupe de cohomologie d'indice $t-s_\xi(\varrho)+1$ de
$\lexp p j^{\geq tg_{-1}}_{\overline{1_h},!*} HT_{\overline \Zm_l,\xi,\overline{1_h}}(\pi_{v,-1},\Pi_t)$.
On part des suites exactes courtes
$$\begin{array}{ccccc}
P_1 & \hookrightarrow & j^{= tg_{-1}}_{\overline{1_h},!} 
HT_{\overline \Zm_l,\xi,\overline{1_h}}(\pi_{v,-1},\Pi_t) & \twoheadrightarrow & 
\lexp p j^{= tg_{-1}}_{\overline{1_h},!*} HT_{\overline \Zm_l,\xi,\overline{1_h}}(\pi_{v,-1},\Pi_t) \\
P_2 & \hookrightarrow & j^{=(t+1)g_{-1}}_{\overline{1_h},!} j^{=(t+1)g_{-1},*}_{\overline{1_h}} P_1 &
\twoheadrightarrow & P_1 \\
\cdots & \cdots & \cdots & \cdots & \cdots \\
P_{s_\xi(\varrho)-t} & \hookrightarrow & j^{=(s_\xi(\varrho)-1)g_{-1}}_{\overline{1_h},!}
 j^{=(s_\xi(\varrho)-1)g_{-1},*}_{\overline{1_h}} P_{s_\xi(\varrho)-t-1} & \twoheadrightarrow &
 P_{s_\xi(\varrho)-t-1}.
 \end{array}$$
Rappelons que ces suites sont construites sur $\overline \Qm_l$ dans \cite{boyer-invent2} et que 
leur version entiËre est le rÈsultat principal de \cite{boyer-duke}.
Comme en outre les $X^{=tg_{-1}}_{\IC,\bar s_v,\overline{1_h}}$ sont affines, on en dÈduit que
\begin{multline*}
H^{t-s_\xi(\varrho)+1} (X^{\geq 1}_{\IC,\bar s_v}, \lexp p j^{=tg_{-1}}_{\overline{1_h},!*}
HT_{\overline \Zm_l,\xi,\overline{1_h}}(\pi_{v,-1},\Pi_t)) \simeq \\
H^{t-s_\xi(\varrho)+2} (X^{\geq 1}_{\IC,\bar s_v}, P_1) \simeq 
\cdots  \simeq H^{0} (X^{\geq 1}_{\IC,\bar s_v}, P_{s_\xi(\varrho)-t-1}).
\end{multline*}
On peut en outre continuer le processus prÈcÈdent de sorte que $P_{s_\xi(\varrho)-t}$
s'inscrit dans une suite exacte courte
$$0 \rightarrow Q \longrightarrow P_{s_\xi(\varrho)-t} \longrightarrow 
\lexp p j^{=s_\xi(\varrho)g_{-1}}_{\overline{1_h},!*} HT_{\overline \Zm_l,\xi,\overline{1_h}}
(\pi_{v,-1},\widetilde{\Pi_{s_\xi(\varrho)}} )\rightarrow 0,$$
o˘ on a posÈ 
$$\widetilde{\Pi_{s_\xi(\varrho)}}:= \Pi_t \{ (s_\xi(\varrho)-t)
\frac{g-1}{2} \} \otimes \speh_{s_\xi(\varrho)-t}(\pi_{v,-1} \{ -t \frac{g-1}{2} \})$$
et o˘ les constituants irrÈductibles de $Q \otimes_{\overline \Zm_l} \overline \Qm_l$
sont de la forme $\lexp p j^{=(s_\xi(\varrho)+\delta)g_{-1}}_{\overline{1_h},!*} 
HT_{\overline \Zm_l,\xi,\overline{1_h}}(\pi_{v,-1},\Pi'_\delta)$ et n'ont donc pas de cohomologie.
On obtient ainsi la suite exacte longue de cohomologie
\begin{multline*}
0 \rightarrow H^{-1} (X^{\geq 1}_{\IC,\bar s_v}, P_{s_\xi(\varrho)-t-1}) \longrightarrow \\
H^{0} (X^{\geq 1}_{\IC,\bar s_v}, P_{s_\xi(\varrho)-t})  \longrightarrow
H^{0} (X^{\geq 1}_{\IC,\bar s_v}, j^{=(s_\xi(\varrho)-1)g_{-1}}_{\overline{1_h},!}
j^{=(s_\xi(\varrho)-1)g_{-1},*}_{\overline{1_h}} P_{s_\xi(\varrho)-t-1}) \\ \longrightarrow 
H^{0} (X^{\geq 1}_{\IC,\bar s_v}, P_{s_\xi(\varrho)-t-1}) \rightarrow 0,
\end{multline*}
avec $H^{0} (X^{\geq 1}_{\IC,\bar s_v}, P_{s_\xi(\varrho)-t}) \simeq H^{0} (X^{\geq 1}_{\IC,\bar s_v},
\lexp p j^{=s_\xi(\varrho)g_{-1}}_{\overline{1_h},!*} HT_{\overline \Zm_l,\xi,\overline{1_h}}
(\pi_{v,-1},\widetilde{\Pi_{s_\xi(\varrho)}} ))$. Posons
$$j^{=(s_\xi(\varrho)-1)g_{-1},*}_{\overline{1_h}} P_{s_\xi(\varrho)-t-1} \simeq 
HT_{\overline \Zm_l,\xi,\overline{1_h}}(\pi_{v,-1},\Pi_{s_\xi(\varrho)-t-1})$$
avec
$$\Pi_{s_\xi(\varrho)-t-1}:=\Pi_t \{ (s_\xi(\varrho)-t-1)\frac{g-1}{2} \} \otimes 
\speh_{s_\xi(\varrho)-t-1}(\pi_{v,-1} \{ -t\frac{g-1}{2} \}).$$
On a alors
$$0 \rightarrow A \longrightarrow  j^{=(s_\xi(\varrho)-1)g_{-1}}_{\overline{1_h},!}
j^{=(s_\xi(\varrho)-1)g_{-1},*}_{\overline{1_h}} P_{s_\xi(\varrho)-t-1} \longrightarrow
\lexp p  j^{=(s_\xi(\varrho)-1)g_{-1}}_{\overline{1_h},!*}
j^{=(s_\xi(\varrho)-1)g_{-1},*}_{\overline{1_h}} P_{s_\xi(\varrho)-t-1} \rightarrow 0$$
avec 
$$\xymatrix{
P_{s_\xi(\varrho)-t} \ar@{^{(}->}[r] \ar@{->>}[d] & A \ar@{->>}[d] \\
\lexp p j^{=s_\xi(\varrho)g_{-1}}_{\overline{1_h},!*} HT_{\overline \Zm_l,\xi,\overline{1_h}}
(\pi_{v,-1},\widetilde{\Pi_{s_\xi(\varrho)}} ) \ar@{^{(}->}[r] &
\lexp p j^{=s_\xi(\varrho)g_{-1}}_{\overline{1_h},!*} HT_{\overline \Zm_l,\xi,\overline{1_h}}
(\pi_{v,-1},\Pi_{s_\xi(\varrho)} )
}$$
avec 
$$\Pi_{s_\xi(\varrho)} \simeq \Pi_t \{ (s_\xi(\varrho)-t)
\frac{g-1}{2} \} \otimes \Bigl ( \speh_{s_\xi(\varrho)-t-1}(\pi_{v,-1} \{ -\frac{1}{2} \}) \times
\pi_{v,-1} \{ \frac{s_\xi(\varrho)-t-2}{2} \} \Bigr ) \{ -t \frac{g-1}{2} \},$$
et o˘ l'inclusion de la ligne du bas est donnÈe par
$$\speh_{s_\xi(\varrho)-t}(\pi_{v,-1}) \hookrightarrow 
\speh_{s_\xi(\varrho)-t-1}(\pi_{v,-1} \{ -\frac{1}{2} \}) \times \pi_v \{ \frac{s_\xi(\varrho)-t-2}{2} \} .$$
Rappelons par ailleurs que la cohomologie de $P_{s_\xi(\varrho)-t}$ (resp. de $A$)
est isomorphe ‡ celle de $\lexp p j^{=s_\xi(\varrho)g_{-1}}_{\overline{1_h},!*} 
HT_{\overline \Zm_l,\xi,\overline{1_h}} (\pi_{v,-1},\widetilde{\Pi_{s_\xi(\varrho)}} )$
(resp. ‡ celle de $ \lexp p j^{=s_\xi(\varrho)g_{-1}}_{\overline{1_h},!*} 
HT_{\overline \Zm_l,\xi,\overline{1_h}}(\pi_{v,-1},\Pi_{s_\xi(\varrho)} )$).
D'aprËs la remarque prÈcÈdente, on en dÈduit alors que la flËche
$$H^{0} (X^{\geq 1}_{\IC,\bar s_v}, P_{s_\xi(\varrho)-t})  \longrightarrow
H^{0} (X^{\geq 1}_{\IC,\bar s_v}, j^{=(s_\xi(\varrho)-1)g_{-1}}_{\overline{1_h},!}
j^{=(s_\xi(\varrho)-1)g_{-1},*}_{\overline{1_h}} P_{s_\xi(\varrho)-t-1})$$
n'est pas stricte et que donc la torsion de 
$$H^{0} (X^{\geq 1}_{\IC,\bar s_v}, P_{s_\xi(\varrho)-t-1}) \simeq 
H^{t-s_\xi(\varrho)+1} (X^{\geq 1}_{\IC,\bar s_v}, \lexp p j^{=tg_{-1}}_{\overline{1_h},!*}
HT_{\overline \Zm_l,\xi,\overline{1_h}}(\pi_{v,-1},\Pi_t))$$
n'est pas nulle.

\end{proof}

Supposons ‡ prÈsent qu'il existe une reprÈsentation irrÈductible cuspidale $\pi_{v,0}$ de $\varrho$-type $0$ telle que
$$0 \leq  s_{\xi}(\pi_{v,0}) - 1 < s_{\xi}(\varrho) - m(\varrho).$$

\begin{prop} \label{prop-torsion0}
Sous les hypothËses prÈcÈdentes et pour un niveau $I$ assez profond, la torsion de
$$H^{m(\varrho)- s_{\xi}(\pi_{v,-1})-1}(X_{\IC,\bar s_v},\lexp {p+} j^{=m(\varrho)g_{-1}}_{!*} 
HT_{\overline \Zm_l}(\pi_{v,0},\Pi_1))$$
est non nulle.
\end{prop}

\begin{proof}
Par dualitÈ, on est ramenÈ ‡ prouver que la torsion de 
\addtocounter{smfthm}{1}
\begin{equation} \label{eq-Hi}
H^{s_{\xi}(\pi_{v,-1})-m(\varrho)}(X_{\IC,\bar s_v},\lexp {p} j^{=g_{0}}_{!*} 
HT_{\overline \Zm_l} (\pi_{v,0},\Pi_1))
\end{equation}
est non nulle, o˘ on a posÈ $g_0=m(\varrho)g_{-1}$ de sorte que $\pi_{v,0}$ est une reprÈsentation de 
$GL_{g_0}(F_v)$. Comme par hypothËse $s_{\xi}(\pi_{v,0}) - 1 < s_{\xi}(\pi_{v,-1}) - m(\varrho)$
il rÈsulte de la proposition \ref{prop-torsion-Ql} que la partie libre de (\ref{eq-Hi}) est nulle et on est
donc ramenÈ ‡ montrer que (\ref{eq-Hi}) est non nul.

\begin{lemm} \label{lem-coho2}
Soient $u \geq 0$ et $\pi_{v,u}$ une reprÈsentation irrÈductible cuspidale de $\varrho$-type $u$.
Pour tout $1 \leq t \leq s_u:= \lfloor \frac{d}{g_u} \rfloor$ et pour tout $i > s_{\xi}(\varrho)-tg_u$,
$$H^{i}(X_{I,\bar s_v},\lexp {p} j^{=tg_{u}}_{!*} HT_{\overline \Zm_l} (\pi_{v,u},\Pi_t))$$ 
est nul.
\end{lemm}

\begin{proof}
On raisonne par rÈcurrence sur $t$ de $s_u$ ‡ $1$. Pour $t=s_u$, on a
$\lexp {p} j^{=s_ug_{u}}_{!*} HT_{\overline \Zm_l} (\pi_{v,u},\Pi_{s_u}) \simeq
j^{=s_ug_{u}}_{!} HT_{\overline \Zm_l} (\pi_{v,u},\Pi_{s_u})$ avec, d'aprËs la proposition
\ref{prop-red-D}, dans le groupe de Grothendieck des $\Fm_l$-faisceaux pervers
$$\Fm \bigl ( j^{=s_ug_{u}}_{!} HT_{\overline \Zm_l} (\pi_{v,u},\Pi_{s_u})\bigr ) =
g_u \Fm \bigl (  j^{=s_um(\varrho)l^ug_{-1}}_{!} HT_{\overline \Zm_l} (\pi_{v,-1},\Pi_{s_u})\bigr ).$$
Le rÈsultat dÈcoule du corollaire \ref{coro-torsion1} en utilisant la suite exacte courte (\ref{eq-sec-torsion0}).

Supposons ‡ prÈsent le rÈsultat acquis jusqu'au rang $t+1 \geq 2$. 
Le mÍme raisonnement que dans le cas $t=s_u$, donne la nullitÈ des 
$H^i(X_{I,\bar s_v},j^{=tg_{u}}_{!}  HT_{\overline \Zm_l} (\pi_{v,u},\Pi_t))$
pour tout $i > s_{\xi}(\pi_{v,-1})-tg_u$. On utilise alors la suite spectrale (\ref{eq-ss-filt1}) avec 
$\pi_v=\pi_{v,u}$:
d'aprËs l'hypothËse de rÈcurrence, tous les $E_1^{p,q}$ sont nuls pour $p>0$ et 
$p+q > s_{\xi}(\varrho)-tg_u$, ainsi donc que les $E_\oo^{p+q}$, de sorte qu'il en est de mÍme des 
$E_1^{0,q}$ pour $q > s_{\xi}(\varrho)-tg_u$, d'o˘ le rÈsultat.
\end{proof}

Reprenons la preuve du lemme ci-dessus pour $u=0$ et $t=1$: on a encore la nullitÈ des 
$H^i((X_{I,\bar s_v},\lexp {p} j^{=g_{0}}_{!*} HT_{\overline \Zm_l} (\pi_{v,0},\Pi_1))$ pour tout 
$i>s_{\xi}(\varrho)-m(\varrho)$ en revanche pour $i=s_{\xi}(\varrho)-m(\varrho)$, comme d'aprËs la 
proposition \ref{prop-torsion-Ql}, la partie libre de 
$$H^{s_{\xi}(\varrho)-m(\varrho)}(X_{\IC,\bar s_v}, j^{=m(\varrho)g_{-1}}_{!} 
HT_{\overline \Zm_l} (\pi_{v,-1},\Pi_{m(\varrho)}))$$ 
est non nulle, on en dÈduit que, pour un niveau $I$ assez profond,
$$H^{s_{\xi}(\varrho)-m(\varrho)}(X_{I,\bar s_v},\Fm j^{=m(\varrho)g_{-1}}_{!} 
HT_{\overline \Zm_l} (\pi_{v,-1},\Pi_{m(\varrho)}))$$
et donc
$$H^{s_{\xi}(\varrho)-m(\varrho)}(X_{I,\bar s_v},\Fm j^{=g_{0}}_{!} 
HT_{\overline \Zm_l} (\pi_{v,0},\Pi_{m(\varrho)}))$$ 
est aussi non nul. Or comme par hypothËse $s_{\xi}(\pi_{v,0}) - 1 < s_{\xi}(\varrho) - m(\varrho)$, il dÈcoule 
de la proposition \ref{prop-torsion-Ql}, que la partie libre de 
$$H^{s_{\xi}(\varrho)-m(\varrho)}(X_{I,\bar s_v}, j^{=g_{0}}_{!} 
HT_{\overline \Zm_l} (\pi_{v,0},\Pi_{m(\varrho)}))$$ 
est nulle, le rÈsultat dÈcoule alors de (\ref{eq-sec-torsion0}).
\end{proof}

\rem on notera que la preuve prÈcÈdente montre l'inÈgalitÈ $s_{\xi}(\pi_{v,0}) - 1 
\leq s_{\xi}(\varrho) - m(\varrho)$ ainsi que le fait que $s_{\xi}(\pi_{v,0})$ ne dÈpend que de $\varrho$.
En ce qui concerne le niveau fini $I$, il suffit qu'il soit suffisamment profond de sorte que
la partie libre de 
$H^{s_{\xi}(\varrho)-m(\varrho)}(X_{\IC,\bar s_v}, j^{=m(\varrho)g_{-1}}_{!} 
HT_{\overline \Zm_l} (\pi_{v,-1},\Pi_{m(\varrho)}))$ soit non nulle, i.e. telle qu'il existe $\Pi$
automorphe $\xi$-cohomologique avec 
\begin{itemize}
\item $\Pi_v$ de la forme $\speh_{s_\xi(\varrho)}(\pi_{v,-1}) \times ?$
pour $\pi_{v,-1}$ de rÈduction modulo $l$ isomorphe ‡ $\varrho$, et

\item ayant des vecteurs invariants non nuls sous $I$.
\end{itemize}

\subsection{Rappels sur les filtrations de stratification de $\Psi_{\IC}$}

Afin de filtrer efficacement $\Psi_\IC$, il est plus commode de commencer
par dÈcouper $\Psi_\IC$ selon ses $\varrho$-facteurs directs, cf. \cite{boyer-entier} proposition 2.2.1:
$$\Psi_{\IC} \simeq \bigoplus_{1 \leq g \leq d}\bigoplus_{\varrho \in \scusp_{F_v}(g)} 
\Psi_{\IC,\varrho}$$
o˘ pour tout $\varrho \in \scusp_{F_v}(g)$, le facteur direct $\Psi_{\IC,\varrho}$ est libre
et regroupe, aprËs tensorisation par $\overline \Qm_l$, tous les faisceaux pervers d'Harris-Taylor
associÈs ‡ une reprÈsentation irrÈductible cuspidale $\pi_v$ de $\varrho$-type $u$ pour $u \geq -1$,
au sens de la dÈfinition \ref{defi-varrho-type}.

Rappelons alors la construction de la filtration de stratification de $\Psi_{\IC,\varrho}$ introduite
dans \cite{boyer-torsion}. Pour $1 \leq h <d$, on note 
$X_{\IC,\bar s_v}^{1 \leq h}:=X_{\IC,\bar s_v}^{\geq 1}-X_{\IC,\bar s_v}^{\geq h+1}$ et
$$j^{1 \leq h}:X_{\IC,\bar s_v}^{1\leq h} \hookrightarrow X_{\IC,\bar s_v}^{\geq 1}.$$
On dÈfinit alors
$$\Fil^r_{!}(\Psi_{\IC,\varrho}):=\im_\FC \Bigl ( \lexp {p+} j^{1 \leq r}_! j^{1 \leq r,*} \Psi_{\IC,\varrho}
 \longrightarrow \Psi_{\IC,\varrho} \Bigr ),$$
o˘ l'image est prise dans la catÈgorie des faisceaux pervers libres 
$$\FC=\lexp p \CC(X_{\IC,\bar s_v},\overline \Zm_l) \cap \lexp {p+} \CC(X_{\IC,\bar s_v},\overline \Zm_l),$$
dÈfinie comme l'intersection des faisceaux qui sont pervers pour les deux $t$-structures $p$ et $p+$.
Dans \cite{boyer-duke}, on montre
\begin{itemize}
\item d'une part que les $\Fil^r_{!}(\Psi_{\IC,\varrho})$ sont en fait l'image dans la catÈgorie 
des $p$-faisceaux pervers et

\item chacun des graduÈs $\gr^r_{!}(\Psi_{\IC,\varrho})$ admet une filtration construite comme suit:
on filtre le faisceau localement constant $j^{=r,*} \gr^r_{!}(\Psi_{\IC,\varrho})$ de faÁon ‡ ce
que les graduÈs soient des systËmes locaux d'Harris-Taylor, puis on applique ‡ cette filtration
le foncteur $j^{=r}_!$ et on prend les images successives dans la catÈgorie des 
faisceaux $p$-pervers. Le rÈsultat principal de \cite{boyer-duke}
est alors que les graduÈs obtenus sont sans torsion et sont des $p$-faisceaux
pervers d'Harris-Taylor.
\end{itemize}

Dualement, cf. \cite{boyer-torsion} 2.2.6, on peut dÈfinir
$$\Fil_{*}^{-r}(\Psi_{\IC,\varrho})=\ker_\FC \bigl ( \Psi_{\IC,\varrho} \longrightarrow
\lexp {p} j^{1 \leq r}_* j^{1 \leq r,*} \Psi_{\IC,\varrho} \bigr ).$$ 
On obtient ainsi une filtration
$$0=\Fil_{*}^{-d}(\Psi_{\IC,\varrho}) \subset \Fil_{*}^{1-d}(\Psi_{\IC,\varrho}) \subset \cdots \subset 
\Fil_{*}^0(\Psi_{\IC,\varrho})=\Psi_{\IC,\varrho}$$
telle que, d'aprËs \cite{boyer-duke} par application de la dualitÈ de Grothendieck-Verdier,
\begin{itemize}
\item les $\Fil_{*}^{-r}(\Psi_{\IC,\varrho})$ sont obtenus comme les noyaux dans la catÈgorie des 
$p+$-faisceaux pervers et

\item chacun des graduÈs $\gr_{*}^{-r}(\Psi_{\IC,\varrho})$ est libre et muni d'une filtration dont les graduÈs
sont des $p+$-faisceaux pervers d'Harris-Taylor.
\end{itemize}

\subsection{Torsion dans la cohomologie de $\Psi_{I,\varrho}$}

Dans la suite on fixe
\begin{itemize}
\item $\varrho$ tel que $m(\varrho)=2$ et

\item une reprÈsentation irrÈductible cuspidale $\pi_v$ de 
$\varrho$-type $-1$ avec $s_\xi(\varrho) \geq 4$ maximal i.e. Ègal ‡ $\lfloor \frac{d}{g_{-1}} \rfloor$.

\item Pour $\Pi$ une reprÈsentation automorphe
irrÈductible $\xi$-cohomologique dont la composante locale en $v$ est de la forme
$\speh_{s_\xi(\varrho)} (\pi_v) \times ?$, on fixe un niveau fini $I$ tel que $\Pi$ possËde des
vecteurs invariants sous $I$.
\end{itemize}

\rem comme $s_\xi(\varrho) \geq 3$ est maximal, la proposition \ref{prop-torsion0} s'applique.  
On notera aussi que pour $\varrho$ le caractËre trivial et $\xi$ la reprÈsentation triviale, les hypothËses
prÈcÈdentes sont vÈrifiÈes.

Comme prÈcÈdemment on note $\Psi_{\IC,\varrho,\xi}:=\Psi_{\IC,\varrho} \otimes V_{\overline \Zm_l,\xi}$.

\begin{theo} \label{theo-equivalence2}
Sous les hypothËses prÈcÈdentes et pour un niveau $I$ assez profond, pour un au moins des indices 
$i \in \{ 2-s_\xi(\varrho), 3-s_\xi(\varrho) \}$,
la torsion de $H^i(X_{I,\bar s_v},\Psi_{I,\varrho,\xi})$ est non nulle.
\end{theo}

\rem en ce qui concerne le niveau fini $I$ de l'ÈnoncÈ, il faut qu'il soit tel qu'il existe $\Pi$
automorphe $\xi$-cohomologique avec 
\begin{itemize}
\item $\Pi_v$ de la forme $\speh_{s_\xi(\varrho)}(\pi_{v,-1}) \times ?$
pour $\pi_{v,-1}$ de rÈduction modulo $l$ isomorphe ‡ $\varrho$, et

\item ayant des vecteurs invariants non nuls sous $I$.
\end{itemize}

\begin{proof}
Supposons dans un premier temps qu'il existe une reprÈsentation irrÈductible cuspidale
$\pi_{v,-1}$ de $\varrho$-type $-1$ telle que la torsion de 
$$H^0(X_{I,\bar s_v,\overline{1_h}}^{\geq (s_\xi(\varrho)-1)g_{-1}}, 
\lexp p j^{\geq (s_\xi(\varrho)-1)g_{-1}}_{\overline{1_h},!*}  
HT_{\overline \Zm_l,\xi,\overline{1_h}}(\pi_{v,-1},\Pi_{s_\xi(\varrho)-1}))$$
est non nulle. Rappelons que
$$\Fil_{!}^1(\Psi_{\IC,\varrho,\xi} \otimes_{\overline \Zm_l} \overline \Qm_l ) \simeq
\bigoplus_{\pi_{v,-1} \in \cusp(\varrho,-1)} \Fil_{!,\pi_{v,-1}}^1(\Psi_{\IC,\varrho,\xi}
\otimes_{\overline \Zm_l} \overline \Qm_l )$$
avec $j^{=g_{-1}}_! HT_{\xi,\overline \Qm_l} (\pi_{v,-1},\pi_{v,-1}) \twoheadrightarrow 
\Fil_{!,\pi_{v,-1}}^1(\Psi_{\IC,\varrho,\xi} \otimes_{\overline \Zm_l} \overline \Qm_l )$.

Reprenons alors, ligne ‡ ligne, la preuve du lemme \ref{lem-equivalence} en remplaÁant 
$\lexp p j^{\geq tg_{-1}}_{\overline{1_h},!*}  HT_{\overline \Zm_l,\xi,\overline{1_h}}(\pi_{v,-1},\Pi_t)$
par $\Fil_{!}^1(\Psi_{\IC,\varrho,\xi})$. D'aprËs \cite{boyer-duke}, on a encore une sÈrie de 
suites exactes courtes
$$\begin{array}{ccccc}
P_1 & \hookrightarrow & j^{= g_{-1}}_{!} j^{=g_{-1},*} \Fil_{!}^1(\Psi_{\IC,\varrho,\xi})
& \twoheadrightarrow & \Fil_{!}^1(\Psi_{\IC,\varrho,\xi}) \\
P_2 & \hookrightarrow & j^{=2g_{-1}}_! j^{=2g_{-1},*} P_1 &\twoheadrightarrow & P_1 \\
\cdots & \cdots & \cdots & \cdots & \cdots \\
P_{s_\xi(\varrho)-1} & \hookrightarrow & j^{=(s_\xi(\varrho)-1)g_{-1}}_{!}
 j^{=(s_\xi(\varrho)-1)g_{-1},*} P_{s_\xi(\varrho)-2} & \twoheadrightarrow & P_{s_\xi(\varrho)-2}.
\end{array}$$
Les strates $X^{=tg_{-1}}_{\IC,\bar s_v}$ Ètant affines, on a toujours
$$H^{2-s_\xi(\varrho)}(X^{\geq 1}_{\IC,\bar s_v},\Fil_!^1(\Psi_{\IC,\varrho})) \simeq
H^{3-s_\xi(\varrho)}(X^{\geq 1}_{\IC,\bar s_v}, P_1) \simeq \cdots \simeq
H^{0}(X^{\geq 1}_{\IC,\bar s_v}, P_{s_\xi(\varrho)-2}).$$
De la suite exacte courte
$$0 \rightarrow Q \otimes_{\overline \Zm_l} \overline \Qm_l \longrightarrow 
P_{s_\xi(\varrho)-1} \otimes_{\overline \Zm_l} \overline \Qm_l  \longrightarrow 
\bigoplus_{\pi_{v,-1} \in \cusp(\varrho,-1)} \lexp p j^{=s_\xi(\varrho)g_{-1}}_{!*} HT_{\overline \Qm_l,\xi}
(\pi_{v,-1},\widetilde{\Pi_{s_\xi(\varrho)}} )\rightarrow 0,$$
avec $\widetilde{\Pi_{s_\xi(\varrho)}}:= \speh_{s_\xi(\varrho)}(\pi_{v,-1})$ et o˘
les constituants irrÈductibles de $Q \otimes_{\overline \Zm_l} \overline \Qm_l$
n'ont comme prÈcÈdemment pas de cohomologie, on en dÈduit la suite exacte longue de cohomologie
\begin{multline*}
0 \rightarrow H^{-1} (X^{\geq 1}_{\IC,\bar s_v}, P_{s_\xi(\varrho)-2}) \longrightarrow \\
H^{0} (X^{\geq 1}_{\IC,\bar s_v}, \lexp p j^{=s_\xi(\varrho)g_{-1}}_{!*}  j^{=s_\xi(\varrho)g_{-1},*}
P_{s_\xi(\varrho)-1} \longrightarrow
H^{0} (X^{\geq 1}_{\IC,\bar s_v}, j^{=(s_\xi(\varrho)-1)g_{-1}}_{!}
j^{=(s_\xi(\varrho)-1)g_{-1},*} P_{s_\xi(\varrho)-2}) \\ \longrightarrow 
H^{0} (X^{\geq 1}_{\IC,\bar s_v}, P_{s_\xi(\varrho)-2}) \rightarrow 0,
\end{multline*}
o˘ 
$$\bigl ( j^{=(s_\xi(\varrho)-1)g_{-1},*} P_{s_\xi(\varrho)-2} \bigr )  \otimes_{\overline \Zm_l} \overline \Qm_l 
\simeq \bigoplus_{\pi_{v,-1} \in \cusp(\varrho,-1)} HT_{\overline \Qm_l,\xi}(\pi_{v,-1},
\speh_{s_\xi(\varrho)-1}(\pi_{v,-1})).$$
On a alors
$$0 \rightarrow A \longrightarrow  j^{=(s_\xi(\varrho)-1)g_{-1}}_{!}
j^{=(s_\xi(\varrho)-1)g_{-1},*} P_{s_\xi(\varrho)-2} \longrightarrow
\lexp p  j^{=(s_\xi(\varrho)-1)g_{-1}}_{!*} j^{=(s_\xi(\varrho)-1)g_{-1},*} P_{s_\xi(\varrho)-2} \rightarrow 0$$
avec 
$$\xymatrix{
P_{s_\xi(\varrho)-1} \ar@{^{(}->}[r] \ar@{->>}[d] & A \ar@{->>}[d] \\
B  \ar@{^{(}->}[r] & C
}$$
et o˘
$$B  \otimes_{\overline \Zm_l} \overline \Qm_l \simeq \bigoplus_{\pi_{v,-1} \in \cusp(\varrho,-1)}
\lexp p j^{=s_\xi(\varrho)g_{-1}}_{\overline{1_h},!*} HT_{\overline \Qm_l,\xi}
(\pi_{v,-1},\widetilde{\Pi_{s_\xi(\varrho)}} )$$
et
$$C  \otimes_{\overline \Zm_l} \overline \Qm_l \simeq \bigoplus_{\pi_{v,-1} \in \cusp(\varrho,-1)}
\lexp p j^{=s_\xi(\varrho)g_{-1}}_{!*} HT_{\overline \Zm_l,\xi} (\pi_{v,-1},\Pi_{s_\xi(\varrho)} ),$$
avec 
$$\Pi_{s_\xi(\varrho)} \simeq \speh_{s_\xi(\varrho)-t-1}(\pi_{v,-1} \{ -\frac{1}{2} \}) \times
\pi_{v,-1} \{ \frac{s_\xi(\varrho)-t-2}{2} \}.$$
On se retrouve alors dans la mÍme configuration de la fin de la preuve du lemme \ref{lem-equivalence}
o˘ la flËche
$$H^{0} (X^{\geq 1}_{\IC,\bar s_v}, \lexp p j^{=s_\xi(\varrho)g_{-1}}_{!*} 
HT_{\overline \Zm_l,\xi} (\pi_{v,-1},\widetilde{\Pi_{s_\xi(\varrho)}} )) \longrightarrow
H^{0} (X^{\geq 1}_{\IC,\bar s_v}, j^{=(s_\xi(\varrho)-1)g_{-1}}_{!}
j^{=(s_\xi(\varrho)-1)g_{-1},*} P_{s_\xi(\varrho)-2}) 
$$
n'est pas stricte et o˘ donc la torsion de 
$$H^{2-s_\xi(\varrho)}(X^{\geq 1}_{\IC,\bar s_v},\Fil_{!,\pi_{v,-1}}^1(\Psi_{\IC,\varrho})) \simeq
H^{0}(X^{\geq 1}_{\IC,\bar s_v}, P_{s_\xi(\varrho)-2})$$
est non nulle. Par ailleurs sur $\overline \Qm_l$, les constituants irrÈductibles de
$\Psi_{\IC,\varrho,\xi}/\Fil_!^1(\Psi_{\IC,\varrho,\xi})$ sont des faisceaux pervers d'Harris-Taylor
de la forme $\PC(t,\pi'_v) (n)$ avec
\begin{itemize}
\item soit $\pi'_v$ est de $\varrho$-type $\geq 0$,

\item et sinon $t>1$.
\end{itemize}
Il rÈsulte alors du corollaire \ref{coro-torsion1} et du fait que pour
toute reprÈsentation $\pi_{v,0}$ de $\varrho$-type $0$, on ait $s_\xi(\pi_{v,0})<s_\xi(\varrho)-1$
que les $H^i(X^{\geq 1}_{\IC,\bar s_v},\Psi_{\IC,\varrho}/\Fil_!^1(\Psi_{\IC,\varrho}))$ sont nuls
pour $i<2-s_\xi(\varrho)$ et sans torsion pour $i=2-s_\xi(\varrho)$, de sorte que la torsion de
$H^{2-s_\xi(\varrho)}(X^{\geq 1}_{\IC,\bar s_v},\Psi_{\IC,\varrho})$, est non nulle.

\medskip

On suppose ‡ prÈsent que pour tout $\pi_{v,-1} \in \cusp(\varrho,-1)$, la torsion de 
$$H^0(X_{I,\bar s_v,\overline{1_h}}^{\geq (s_\xi(\varrho)-1)g_{-1}}, 
\lexp p j^{\geq (s_\xi(\varrho)-1)g_{-1}}_{\overline{1_h},!*}  
HT_{\overline \Zm_l,\xi,\overline{1_h}}(\pi_v,\Pi_{s_\xi(\varrho)-1}))$$
est nulle de sorte que d'aprËs la proposition \ref{prop-equivalence0}, il en est de 
mÍme pour tous les 
$$H^0(X_{I,\bar s_v}^{\geq tg_{-1}}, \lexp p j^{\geq tg_{-1}}_{!*}  
HT_{\overline \Zm_l,\xi}(\pi_v,\Pi_t)).$$

\begin{lemm}
La torsion de 
$H^{s_\xi(\varrho)-1}(X^{\geq 1}_{I,\bar s_v},\Fil_!^1(\Psi_{\IC,\varrho,\xi}))$ est nulle.
\end{lemm}

\begin{proof}
On reprend les suites exactes courtes de la preuve du thÈorËme \ref{theo-equivalence2} que
l'on Ècrit sous la forme de triangles distinguÈs comme suit
$$
j^{= g_{-1}}_{!}  j^{=g_{-1},*} \Fil_!^1(\Psi_{\IC,\varrho,\xi})  \longrightarrow 
\Fil_!^1(\Psi_{\IC,\varrho,\xi})  \longrightarrow P_1[1] \leadsto $$
$$ j^{=2g_{-1}}_{!} j^{=2g_{-1},*} P_1[1] \longrightarrow  P_1[1] \longrightarrow P_2[2] \leadsto$$
$$ \cdots \cdots$$
$$j^{=(s_\xi(\varrho)-1)g_{-1}}_{!} j^{=(s_\xi(\varrho)-1)g_{-1},*} P_{s_\xi(\varrho)-2}[s_\xi(\varrho)-2] 
\longrightarrow P_{s_\xi(\varrho)-2}[s_\xi(\varrho)-2] \longrightarrow 
 P_{s_\xi(\varrho)-1}[s_\xi(\varrho)-1] \leadsto$$
On obtient alors une suite spectrale 
$$E_2^{p,q} \Rightarrow H^{p+q}(X^{\geq 1}_{\IC,\bar s_v},\Fil^1_!(\Psi_{\IC,\varrho,\xi}))$$
o˘ en posant $P_0=j^{=g_{-1}}_!  j^{=g_{-1},*} \Fil_!^1(\Psi_{\IC,\varrho,\xi})$ et en notant que
la cohomologie de $P_{s_\xi(\varrho)-1}[s_\xi(\varrho)-1]$ est, d'aprËs \ref{coro-torsion1}, Ègale ‡ celle de 
$j^{=s_\xi(\varrho)g_{-1}}_{!} j^{=s_\xi(\varrho)g_{-1},*} P_{s_\xi(\varrho)-1}[s_\xi(\varrho)-1]$,
les $E_2^{p,q}$ sont nuls sauf si $0 \leq q \leq s_\xi(\varrho)-1$ auquel cas on a
$$E_2^{p,q}=H^{p+2q}(X^{\geq 1}_{\IC,\bar s_v},j^{=(q+1)g_{-1}}_!j^{=(q+1)g_{-1},*} P_q).$$
Par hypothËse et en utilisant le lemme \ref{lem-equivalence}, la torsion de
$$H^{2-s_\xi(\varrho)}(X_{I,\bar s_v}^{\geq g_{-1}}, 
\lexp p j^{\geq g_{-1}}_{!*}  HT_{\overline \Zm_l,\xi}(\pi_v,\pi_v))$$
est nulle alors par dualitÈ la torsion de $H^{s_\xi(\varrho)-1}(X_{I,\bar s_v}^{\geq g_{-1}}, 
\lexp p j^{\geq g_{-1}}_{!*}  HT_{\overline \Zm_l,\xi}(\pi_v,\pi_v))$ est nulle de sorte, qu'en utilisant le dÈbut 
de la preuve de la proposition \ref{prop-equivalence0}, la torsion de 
$H^{s_\xi(\varrho)-1}(X_{I,\bar s_v}^{\geq g_{-1}}, j^{\geq g_{-1}}_! HT_{\overline \Zm_l,\xi}(\pi_v,\pi_v))$
est nulle, et donc puisque
$$\bigl ( j^{=g_{-1},*} P_0=j^{=g_{-1},*} \Fil^1_!(\Psi_{\IC,\varrho,\xi}) \bigr )  \otimes_{\overline \Zm_l} 
\overline \Qm_l \simeq \bigoplus_{\pi_{v,-1} \in \cusp(\varrho,-1)} 
HT_{\overline \Qm_l,\xi}(\pi_{v,-1},\pi_{v,-1}),$$
$E_2^{s_\xi(\varrho)-1,0}$ est libre.
On conclut en notant que, d'aprËs le corollaire \ref{coro-torsion1}, les $E_2^{p,q}$ pour $q \geq 1$ et
$p+q=s_\xi(\varrho)-2$ sont nuls.

\end{proof}

\begin{lemm} Pour tout sous-$GL_d(F_v)$-module $\Pi_v$ de 
$H^{s_\xi(\varrho)-2}(X^{\geq 1}_{\IC,\bar s_v},\Fil^1_!(\Psi_{\IC,\varrho,\xi}))$,
le support supercuspidal de sa rÈduction modulo $l$ ne contient aucun segment de
Zelevinsky de longueur $s_\xi(\varrho)$ relativement ‡ $\varrho$.
\end{lemm}

\begin{proof}
Rappelons que sur $\overline \Qm_l$, on a une surjection
$$\Fil^1_!(\Psi_{\IC,\varrho,\xi}) \otimes_{\overline \Zm_l} \overline \Qm_l 
\twoheadrightarrow \bigoplus_{\pi_{v,-1} \in \cusp(\varrho,-1)}
\lexp p j^{=g_{-1}}_{!*} HT_{\overline \Qm_l,\xi}(\pi_{v,-1},\pi_{v,-1})$$
o˘ tous les constituants du noyau sont des faisceaux pervers d'Harris-Taylor de la forme
$HT_{\overline \Qm_l,\xi}(\pi_{v,-1},\st_t(\pi_{v,-1})) \otimes \Xi^{\frac{t-1}{2}}$ avec $t \geq 1$. 
Ainsi d'aprËs le lemme \ref{lem-coho1},
$H^{s_\xi(\varrho)-2}(X^{\geq 1}_{\IC,\bar s_v},\Fil^1_!(\Psi_{\IC,\varrho,\xi}))$ admet une filtration
dont les graduÈs sont de la forme
$H^{s_\xi(\varrho)-2}(X^{\geq 1}_{\IC,\bar s_v},\lexp p j^{=g_{-1}}_{!*} 
HT_{\overline \Zm_l,\xi}(\pi_{v,-1},\pi_{v,-1}))$.
En particulier, il dÈcoule de l'hypothËse faite plus haut, que la torsion de 
$H^{s_\xi(\varrho)-2}(X^{\geq 1}_{\IC,\bar s_v},\Fil^1_!(\Psi_{\IC,\varrho,\xi}))$
est nulle de sorte que la propriÈtÈ se lit sur $\overline \Qm_l$. Or sur $\overline \Qm_l$,
$H^{s_\xi(\varrho)-2}(X^{\geq 1}_{\IC,\bar s_v},\lexp p j^{=g_{-1}}_{!*} 
HT_{\overline \Qm_l,\xi}(\pi_{v,-1},\pi_{v,-1}))$
est de la forme $\Pi^{\oo,v} \otimes \bigl ( \speh_{s_\xi(\varrho)-1}(\pi_{v,-1}) \times ?)$.
Le rÈsultat dÈcoule alors de la maximalitÈ de $s_\xi(\varrho)$ et du fait que
le support cuspidal de $\speh_{s_\xi(\varrho)-1}(\varrho)$ est disjoint
de celui de $\speh_{s_\xi(\varrho)}(\varrho)$.
\end{proof}

Rappelons que $\Psi_{\IC,\varrho}/\Fil_!^1(\Psi_{\IC,\varrho,\xi})$ est ‡ support dans
$X^{\geq 2g_{-1}}_{\IC,\bar s_v}$ et que puisque $m(\varrho)=2$,
\begin{multline*}
\bigl ( j^{=2g_{-1},*} \Psi_{\IC,\varrho,\xi}/\Fil_!^1(\Psi_{\IC,\varrho,\xi}) \bigr ) \otimes_{\overline \Zm_l}
\overline \Qm_l \simeq \bigoplus_{\pi_{v,-1} \in \cusp(\varrho,-1)} HT(\pi_{v,-1},\st_2(\pi_{v,-1}))
\otimes \Xi^{-1/2} \\
\oplus \bigoplus_{\pi_{v,0} \in \cusp(\varrho,0)} HT(\pi_{v,0},\pi_{v,0}).
\end{multline*}
On considËre alors une filtration entiËre 
$$A_0 \hookrightarrow  j^{=2g_{-1},*}  \Psi_{\IC,\varrho,\xi}/\Fil_!^1(\Psi_{\IC,\varrho,\xi}) 
\twoheadrightarrow A_{-1}$$
avec $A_0$ et $A_{-1}$ libres et tels que
$$A_0 \otimes_{\overline \Zm_l} \overline \Qm_l \simeq \bigoplus_{\pi_{v,0} \in \cusp(\varrho,0)} 
HT(\pi_{v,0},\pi_{v,0}),$$
et on note $P_0$ l'image de 
$$j^{=2g_{-1}}_! A_0 \hookrightarrow j^{=2g_{-1}}_! j^{=2g_{-1},*}  
\Psi_{\IC,\varrho,\xi}/\Fil_!^1(\Psi_{\IC,\varrho,\xi}) \longrightarrow \Psi_{\IC,\varrho,\xi}
/\Fil_!^1(\Psi_{\IC,\varrho,\xi}).$$
D'aprËs \cite{boyer-duke}, 
$$0 \rightarrow P_0  \longrightarrow \Psi_{\IC,\varrho,\xi}/\Fil_!^1(\Psi_{\IC,\varrho,\xi}) \longrightarrow
P \rightarrow 0$$
est une suite exacte courte de faisceaux pervers libres. D'aprËs la proposition \ref{prop-torsion0},
la torsion de $H^{s_\xi(\varrho)-2}(X^{\geq 1}_{\IC,\bar s_v},P_0)$ est non nulle et sa rÈduction modulo
$l$ est de la forme $\varrho_0 \times \speh_{s_\xi(\varrho)-2}(\varrho) \times ?$ et contient
donc un segment de Zelevinsky de longueur $s_\xi(\varrho)$ relativement ‡ $\varrho$.
Il rÈsulte alors des deux lemmes prÈcÈdents que cette torsion est encore un sous-module
de torsion de  $H^{s_\xi(\varrho)-2} (X^{\geq 1}_{\IC,\bar s_v},Q)$ o˘ $Q$ est dÈfini comme
le produit fibrÈ suivant:
$$\xymatrix{
\Fil^1_!(\Psi_{\IC,\varrho,\xi}) \ar@{^{(}->}[r] \ar@{=}[d] & Q \ar@{^{(}-->}[d] \ar@{-->>}[r] &
P_0 \ar@{^{(}->}[d] \\
\Fil^1_!(\Psi_{\IC,\varrho,\xi}) \ar@{^{(}->}[r]  & \Psi_{\IC,\varrho,\xi} \ar@{->>}[r]  \ar@{->>}[d] 
& \Psi_{\IC,\varrho,\xi}/ \Fil^1_!(\Psi_{\IC,\varrho,\xi}) \ar@{->>}[d] \\
& P \ar@{=}[r] & P
}$$
En ce qui concerne $P$, il admet une filtration de stratification dont les constituants irrÈductibles,
sur $\overline \Qm_l$, sont ceux de $\Psi_{\IC,\varrho,\xi}$ auxquels on enlËve en particulier ceux de la
forme $\lexp p j^{g}_{!*} HT(\pi_{v},\pi_{v}))$ pour $\pi_v \in \cusp(\varrho,-1) \cup \cusp(\varrho,0)$.
On s'intÈresse alors
‡ $H^{s_\xi(\varrho)-3}(X^{\geq 1}_{\IC,\bar s_v},P)$.
CommenÁons par son quotient libre que l'on peut Ècrire sous la forme
$$H^{s_\xi(\varrho)-3}(X^{\geq 1}_{\IC,\bar s_v},P)=\bigoplus_{\Pi} H[\Pi]$$
o˘ $\Pi$ dÈcrit les classes d'Èquivalences des reprÈsentations automorphes cohomologiques pour
$\xi$ et o˘ $H[\Pi]$ dÈsigne la composante isotypique correspondante. La composante locale en
$v$ d'une telle reprÈsentation automorphe est de la forme $\speh_s(\pi_v)$ avec 
$\pi_v$ une reprÈsentation tempÈrÈe, i.e. de la forme $\st_{t_1}(\pi_{v,1}) \times \cdots \times
\st_{t_r}(\pi_{v,r})$ avec pour $i=1,\cdots,r$, $\pi_{v,i}$ irrÈductible cuspidale et o˘ on rappelle que 
$\speh_s(\pi_v)$
correspond, via la correspondance de Langlands locale ‡ la somme directe $\sigma(\frac{1-s}{2}) \oplus
\cdots \oplus \sigma (\frac{s-1}{2})$ o˘ $\sigma$ correspond ‡ $\pi_v$. Pour que $H[\Pi]$
soit non nulle il faut, d'aprËs \cite{boyer-compositio}, que $s \geq s_\xi(\varrho)-2$. Par ailleurs
comme rappelÈ au lemme \ref{lem-coho1}, ne contribuent ‡ $H[\Pi]$ que la cohomologie de
\begin{itemize}
\item[(i)] $HT(\pi_{v,-1},\st_3(\pi_{v,-1}))$ pour $\Pi$ tel que $\Pi_v \simeq \speh_{s_\xi(\varrho)}(\pi_v)$
et $\pi_v \simeq \pi_{v,-1} \times ?$ avec $\pi_{v,-1} \in \cusp(\varrho,-1)$;

\item[(ii)] $HT(\pi_{v,-1},\st_2(\pi_{v,-1}))$ pour $s=s_\xi(\varrho)-1$.
\end{itemize}
Or pour les contributions de (i) disparaissent dans $H[\Pi]$ alors que celles de (ii) ont, aprËs
rÈduction modulo $l$, un support supercuspidal disjoint de celui de $\speh_{s_\xi(\varrho)}(\varrho)$
et ne nous intÈressent pas.

Enfin en ce qui concerne la torsion de $H^{s_\xi(\varrho)-3}(X^{\geq 1}_{\IC,\bar s_v},P)$,
d'aprËs ce qui prÈcËde, elle admet une filtration dont tous les graduÈs sont des quotients issus des
contributions (i) prÈcÈdentes, i.e. en tant que $GL_d(F_v)$-module, ce sont des quotients de 
$[\overrightarrow{s_\xi(\varrho)-3},\overleftarrow{2}]_{\pi_{v,-1}} \times ?$.
D'aprËs le lemme suivant, la rÈduction modulo $l$ de ces graduÈs
ne contient jamais $\varrho_0 \times \speh_{s_\xi(\varrho)-2}(\varrho) \times ?$ de sorte que la torsion
de $H^{s_\xi(\varrho)-2}(X^{\geq 1}_{\IC,\bar s_v},\Psi_{\IC,\varrho,\xi})$ est bien non nulle.

\end{proof}

\begin{lemm}
Pour tout $s \geq 4$, la rÈduction modulo $l$ de 
$[\overrightarrow{s_\xi(\varrho)-4},\overleftarrow{3}]_{\pi_{v,-1}}$
n'admet pas comme sous-quotient $\speh_{s_\xi(\varrho)-2}(\varrho) \times \varrho_0$.
\end{lemm}

\begin{proof}
Rappelons que par hypohtËse $\varrho \not \simeq \varrho \{ 1 \}$ et $\varrho \simeq \varrho \{ 2 \}$.
Notons $\varrho':= \varrho \{ \frac{s_\xi(\varrho)-1}{2} \}$. Comme
$ \speh_{s_\xi(\varrho)-2}(\varrho) \times \varrho_0$ est aussi irrÈductible et donc isomorphe
‡ la mÍme induite mais relativement au sous-groupe de Levi opposÈ, 
de la rÈciprocitÈ de Frobenius, on en dÈduit une surjection
$$J_{P^{op}_{(s_\xi(\varrho)-2)g_{-1},s_x(\varrho)g_{-1}}(F_v)}(\speh_{s_\xi(\varrho)-2}(\varrho) 
\times \varrho_0) \twoheadrightarrow \speh_{s_\xi(\varrho)-2}(\varrho) \otimes \varrho_0,$$
et donc en composant avec $J_{P^{op}_{g_{-1},(s_\xi(\varrho)-2)g_{-1},s_\xi(\varrho)g_{-1}}(F_v)}$,
il existe une reprÈsentation irrÈductible $\tilde \varrho$ que l'on ne souhaite pas prÈciser, telle que
$\varrho' \otimes \tilde \varrho$ est un constituant de $J_{P^{op}_{g_{-1},s_\xi(\varrho)g_{-1}}(F_v)}
(\speh_{s_\xi(\varrho)-2}(\varrho) \times \varrho_0)$.

En revanche tous les constituants irrÈductibles de $J_{P^{op}_{g_{-1},s_\xi(\varrho)g_{-1}}(F_v)}
([\overrightarrow{s_\xi(\varrho)-4},\overleftarrow{3}]_{\pi_{v,-1}})$ sont de la forme
$\pi_{v,-1} \{ \frac{s_\xi(\varrho)-1}{2} - 3 \}$ et donc, par compatibilitÈ de la rÈduction modulo $l$
avec les foncteurs de Jacquet, et comme $\varrho' \not \simeq \varrho' \{ 1 \}=
r_l(\pi_{v,-1} \{ \frac{s_\xi(\varrho)-1}{2} - 3 \})$, on en dÈduit bien que 
$\speh_{s_\xi(\varrho)-2}(\varrho) \times \varrho_0$ n'est pas un sous-quotient de la rÈduction
modulo $l$ de $[\overrightarrow{s_\xi(\varrho)-4},\overleftarrow{3}]_{\pi_{v,-1}}$.

\end{proof}

\section{Congruences automorphes}
\label{para-principal}

\renewcommand{\theequation}{\arabic{section}.\arabic{smfthm}}
\renewcommand{\thesmfthm}{\arabic{section}.\arabic{smfthm}}

Dans ce paragraphe, le niveau fini $I$ considÈrÈ sera supposÈ maximal ‡ la place $v$.

\subsection{AlgËbres de Hecke}

\begin{nota} \label{nota-spl2}
Pour $I \in \IC$ un niveau fini, soit
$$\Tm_I:=\overline \Zm_l \bigl [T_{w,i}:~w \in \Spl(I) \hbox{ et } i=1,\cdots,d \bigr ],$$
l'algËbre de Hecke associÈe ‡ $\Spl(I)$, o˘ 
 $T_{w,i}$ est la fonction caractÈristique de
$$GL_d(\OC_w) \diag(\overbrace{\varpi_w,\cdots,\varpi_w}^{i}, \overbrace{1,\cdots,1}^{d-i} ) 
GL_d(\OC_w) \subset  GL_d(F_w).$$
\end{nota}

Pour tout $w \in \Spl(I)$, on note
$$P_{\mathfrak{m},w}(X):=\sum_{i=0}^d(-1)^i q_w^{\frac{i(i-1)}{2}} \overline{T_{w,i}} X^{d-i} \in \overline 
\Fm_l[X]$$
le polynÙme de Hecke associÈ ‡ $\mathfrak m$ et 
$$S_{\mathfrak{m}}(w) := \bigl \{ \lambda \in \Tm_I/\mathfrak m \simeq \overline \Fm_l \hbox{ tel que }
P_{\mathfrak{m},w}(\lambda)=0 \bigr \} ,$$
le multi-ensemble des paramËtres de Satake modulo $l$ en $w$ associÈs ‡ $\mathfrak m$.
Avec les notations prÈcÈdentes, l'image $\overline{T_{w,i}}$ de $T_{w,i}$ dans  $\Tm_I/\mathfrak m$
s'Ècrit 
$$\overline{T_{w,i}}=q_w^{\frac{i(1-i)}{2}} \sigma_i (\lambda_1,\cdots, \lambda_d)$$
o˘ $S_{\mathfrak m}(w)=\{ \lambda_1,\cdots,\lambda_d \}$ et $\sigma_i$ dÈsigne la $i$-Ëme fonction
symÈtrique ÈlÈmentaire. 

\begin{nota} \label{nota-mdual}
On notera alors $\mathfrak m^\vee$ l'idÈal maximal de $\Tm_I$ dÈfini par
$$T_{w,i} \in \Tm_I \mapsto q_w^{\frac{i(1-i)}{2}} \sigma_i (\lambda_1^{-1},\cdots,\lambda_d^{-1})
\in \overline \Fm_l.$$
\end{nota}

On reprend ‡ prÈsent les arguments de \cite{boyer-imj}.

\begin{defi}
On dira d'un $\Tm_{I}$-module $M$ qu'il vÈrifie la propriÈtÈ \textbf{(P)},
s'il admet une filtration finie 
$$(0)=\Fil^0(M) \subset \Fil^1(M) \cdots \subset \Fil^r(M)=M$$
telle que pour tout $k=1,\cdots,r$, il existe
\begin{itemize}
\item une reprÈsentation automorphe $\Pi_k$ irrÈductible et entiËre
de $G(\Am)$, apparaissant dans la cohomologie de $X_{\IC,\overline \eta_v}$ ‡ coefficients dans 
$V_{\xi,\overline \Qm_l}$, et telle que la composante locale $\Pi_{k,v}$ de $\Pi_k$ est ramifiÈe,
i.e. $(\Pi_{k,v})^{GL_d(\OC_v)}=(0)$;

\item une reprÈsentation irrÈductible entiËre et non ramifiÈe $\widetilde \Pi_{k,v}$ de mÍme support 
cuspidal que $\Pi_{k,v}$

\item et un $\Tm_{I}$-rÈseau stable $\Gamma$ de 
$(\Pi^{\oo,v}_k)^{I^v} \otimes \widetilde \Pi_{k,v}^{GL_d(\OC_v)}$ tel que
\begin{itemize}
\item soit $\gr^k(M)$ est libre et isomorphe ‡ $\Gamma$,

\item soit $\gr^k(M)$ est de torsion et un sous-quotient de $\Gamma/\Gamma'$ pour 
$\Gamma' \subset \Gamma$ un deuxiËme $\Tm_I$-rÈseau stable.
\end{itemize}
\end{itemize}
Le graduÈ $\gr^k(M)$ sera dit de type $i$ si en outre $\Pi_{k,v}$ est de la forme
$$\st_i(\chi) \times \chi_1 \times \cdots \times \chi_{d-i}$$
o˘ $\chi,\chi_1,\cdots,\chi_{d-i}$ sont des caractËres non ramifiÈs.
\end{defi}

\rem La propriÈtÈ \textbf{(P)} est clairement stable par extensions, par sous-quotients et, en remplaÁant la 
condition $\xi$-cohomologique par $\xi^\vee$-cohomologique, par dualitÈ.

\begin{lemm} Soient $h \geq 1$ et $M$ un sous-quotient irrÈductible de
$H^{d-h}_c(X^{=h}_{I,\bar s_v},V_{\xi,\overline \Qm_l})$
(resp. de $H^{d-h}(X^{\geq h}_{I,\bar s_v},V_{\xi,\overline \Qm_l})$). Alors 
\begin{itemize}
\item soit $M$ vÈrifie la propriÈtÈ \textbf{(P)} et est alors de type $h$ ou $h+1$ (resp. de type $h$)

\item soit $M$ n'est pas un sous-quotient de 
$H^{d-h-1}(X^{\geq h+1}_{I,\bar s_v},V_{\xi,\overline \Qm_l})$.
\end{itemize}
\end{lemm}

\begin{proof}
Le rÈsultat dÈcoule des calculs explicites de ces $\overline \Qm_l$-groupes de cohomologie
en niveau infini donnÈs dans \cite{boyer-compositio} et on renvoie le lecteur ‡ la prÈsentation
qui en est faite au \S 3.3 (resp. au \S 3.2) de \cite{boyer-aif}. PrÈcisÈment d'aprËs loc. cit.,
pour $\Pi^\oo$ une reprÈsentation irrÈductible de $G(\Am^\oo)$, la composante isotypique 
$H^{d-h}_c(X^{=h}_{\IC,\bar s_v},V_\xi \otimes_{\overline \Zm_l} \overline \Qm_l) \{ \Pi^{\oo,v} \}$
est nulle  si $\Pi^\oo$ n'est pas la composante hors de $\oo$ d'une reprÈsentation automorphe
$\xi$-cohomologique $\Pi$. Sinon, on distingue trois situations pour la composante locale en $v$ de $\Pi$:
\begin{itemize}
\item[(i)] soit $\Pi_v \simeq \st_r(\chi_v) \times \pi'_v$ avec $h \leq r \leq d$,

\item[(ii)] soit $\Pi_v \simeq \speh_r(\chi_v) \times \pi'_v$ avec $h \leq r \leq d$,

\item[(iii)] soit $\Pi_v$ n'est pas d'une des deux formes prÈcÈdentes,
\end{itemize}
o˘ dans ce qui prÈcËde, $\chi_v$ est un caractËre non ramifiÈ de $F_v^\times$ et $\pi'_v$ 
une reprÈsentation irrÈductible admissible de $GL_{d-h}(F_v)$.
Alors cette composante isotypique, en tant que reprÈsentation de $GL_d(F_v) \times \Zm$, est
de la forme 
\begin{itemize}
\item dans le cas (i), ‡ $\bigl ( \speh_h(\chi \{ \frac{h-r}{2} \} ) \times \st_{r-h}(\chi \{ \frac{h}{2} \} ) \bigr ) 
\times \pi'_v \otimes \Xi^{\frac{r-h}{2}}$ (resp. nulle si $r \neq h$ et isomorphe ‡
$\st_h(\chi_v)$ sinon);

\item nulle dans le cas (ii) si $r \neq h$ et sinon isomorphe ‡ $\speh_h(\chi_v) \times \pi'_v$.

\item Enfin, dans le cas (iii), la reprÈsentation de $GL_d(F_v)$ obtenue n'aura pas d'invariants sous
$GL_d(\OC_v)$.
\end{itemize}
Ainsi en prenant les invariants sous $I$, dans le cas (i) on obtient un $\Tm_I$-module vÈrifiant
la propriÈtÈ \textbf{(P)} et qui est de type $h$ ou $h+1$. Dans le cas (ii), le $\Tm_I$-module
obtenu n'est pas un sous-quotient de $H^{d-h-1}(X^{\geq h+1}_{I,\bar s_v},V_{\xi,\Qm_l})$ et le cas (iii)
ne fournit rien.

\end{proof}

\subsection{Torsion dans la cohomologie et congruences automorphes}

Supposons ‡ prÈsent qu'il existe $i$ tel que la torsion de $H^i(X_{I,\bar \eta_v},V_\xi)$ est non nulle.
Soit alors $1 \leq h \leq d$ maximal tel qu'il existe $i$ pour lequel le sous-espace de torsion
$H^i(X^{\geq h}_{I,\bar s_v},V_\xi)_{tor}$ de $H^i(X^{\geq h}_{I,\bar s_v},V_\xi)$ est non nul.
Notons que
\begin{itemize}
\item comme $X^{\geq d}_{I,\bar s_v}$ est de dimension nulle, alors $h<d$;

\item d'aprËs le changement de base lisse, $H^i(X_{I,\bar \eta_v},V_\xi) \simeq H^i(X^{\geq 1}_{I,\bar s_v},
V_\xi)$ de sorte que $h \geq 1$.
\end{itemize}

\begin{lemm} Avec les notations ci-dessus, le plus petit indice $i$ tel que la torsion de 
$H^i(X^{\geq h}_{I,\bar s_v},V_\xi)$ est non nulle est $i=0$ et alors tout sous-module irrÈductible
de $H^0(X^{\geq h}_{I,\bar s_v},V_\xi)_{tor}$ vÈrifie la propriÈtÈ \textbf{(P)} en Ètant de type $h+1$.
\end{lemm}

\rem le lemme et le suivant dÈcoulent du lemme \ref{lem-equivalence}; nous allons cependant 
en donner une preuve directe dans le cas de bonne rÈduction.

\begin{proof}
ConsidÈrons la suite exacte courte de faisceaux pervers 
\addtocounter{smfthm}{1}
\begin{equation} \label{eq-sec00}
0 \rightarrow i_{h+1,*} V_{\xi,\overline \Zm_l,|X_{I,\bar s_v}^{\geq h+1}}[d-h-1] \longrightarrow 
j^{\geq h}_! j^{\geq h,*} V_{\xi,\overline \Zm_l,|X^{\geq h}_{I,\bar s_v}}[d-h] \longrightarrow 
V_{\xi,\overline \Zm_l,|X^{\geq h}_{I,\bar s_v}}[d-h] \rightarrow 0.
\end{equation}
Les strates $X^{\geq h}_{I,\bar s_v}$ 
Ètant lisses et $j^{\geq h}$ Ètant affine, les trois termes de la suite exacte sont pervers et sont libres
au sens de la thÈorie de torsion naturelle issue de la structure $\overline \Zm_l$-linÈaire, cf. 
\cite{boyer-torsion} \S 1.1-1.3. Par ailleurs
il rÈsulte du thÈorËme d'Artin, cf. par exemple le thÈorËme 4.1.1 de \cite{BBD} et, donc, de 
l'affinitÈ des strates $X^{=h}_{I,\bar s_v}$, que 
les $H^i(X_{I,\bar s_v}^{\geq h},j^{\geq h}_! j^{\geq h,*} 
V_{\xi,\overline \Zm_l,|X^{\geq h}_{I,\bar s_v}}[d-h])$ sont nuls pour $i<0$ et sans torsion pour $i=0$,
de sorte que pour $i>0$, on a
\addtocounter{smfthm}{1}
\begin{equation} \label{eq-sec}
0 \rightarrow H^{-i-1}(X^{\geq h}_{I,\bar s_v},V_{\xi,\overline \Zm_l}[d-h]) \longrightarrow
H^{-i}(X^{\geq h+1}_{I,\bar s_v},V_{\xi,\overline \Zm_l}[d-h-1]) \rightarrow 0,
\end{equation}
et pour $i=0$,
\addtocounter{smfthm}{1}
\begin{multline} \label{eq-sec0}
0 \rightarrow H^{-1}(X^{\geq h}_{I,\bar s_v},V_{\xi,\overline \Zm_l}[d-h]) \longrightarrow
H^{0}(X^{\geq h+1}_{I,\bar s_v},V_{\xi,\overline \Zm_l}[d-h-1]) \longrightarrow \\
H^0(X^{\geq h}_{I,\bar s_v},j^{\geq h}_! j^{\geq h,*} V_{\xi,\overline \Zm_l}[d-h] ) \longrightarrow
H^{0}(X^{\geq h}_{I,\bar s_v},V_{\xi,\overline \Zm_l}[d-h]) \rightarrow \cdots
\end{multline}
Ainsi si la torsion de $H^i(X^{\geq h}_{I,\bar s_v},V_\xi)$ est non nulle alors $i \geq 0$ et, en utilisant
la dualitÈ de Grothendieck-Verdier, le plus petit tel indice est nÈcessairement $i=0$. Par ailleurs la
torsion de $H^0(X^{\geq h}_{I,\bar s_v},V_\xi)$ se relËve ‡ la fois dans 
$H^0_c(X^{=h}_{I,\bar s_v},V_\xi)$ et $H^0(X^{\geq h+1}_{I,\bar s_v},V_\xi)$, lesquels sont libres.
Ainsi d'aprËs le lemme prÈcÈdent, la torsion de $H^0(X^{\geq h}_{I,\bar s_v},V_\xi)$ vÈrifie
la propriÈtÈ \textbf{(P)} en Ètant de type $h+1$.
\end{proof}

\begin{lemm}
Avec les notations prÈcÈdentes, pour tout $1 \leq h' \leq h$, l'indice $h-h'$ est le plus grand $i$ tel que la
torsion de $H^{-i}(X^{\geq h'}_{I,\bar s_v},V_\xi)_{tor}$ est non nulle. En outre celle-ci vÈrifie la propriÈtÈ
\textbf{(P)} en Ètant de type $h+1$.
\end{lemm}

\begin{proof}
On raisonne par rÈcurrence sur $h'$ de $h$ ‡ $1$. Le cas $h'=h$ est traitÈ dans le lemme prÈcÈdant,
supposons donc le rÈsultat acquis jusqu'au rang $h'+1$ et traitons le cas de $h'$. On reprend
alors les suites exactes (\ref{eq-sec00}) avec $h'$. Le rÈsultat dÈcoule alors de (\ref{eq-sec}) 
et de l'hypothËse de rÈcurrence.
\end{proof}

Le cas $h'=1$ du lemme prÈcÈdent, en utilisant le thÈorËme de changement de base lisse, fournit alors
l'ÈnoncÈ suivant.

\begin{prop} \label{prop-p1}
Soit $i$ maximal, s'il existe, tel que la torsion de $H^{-i}(X_{I,\bar \eta_v},V_\xi)$ est non nul.
Il vÈrifie alors la propriÈtÈ \textbf{(P)} en Ètant de type $i+2$.
\end{prop}

\begin{coro} \label{coro-principal}
Soit $\mathfrak m$ un idÈal maximal de $\Tm_I$ et $i$ maximal s'il existe, tel que
la torsion de $H^{-i}(X_{I,\bar \eta},V_\xi)_{\mathfrak m}$ est non nulle. Il existe alors une collection
$\Pi(w)$ indexÈe par les places $w$ de $F$ divisant une place fixÈe $u$ de $E$ au dessus d'un premier
$p \in \Spl(I)$, de reprÈsentations irrÈductibles automorphes $\xi$-cohomologiques
telles que:
\begin{itemize}
\item pour tout $q \in \Spl(I)$ distinct de $p$ (resp. $q=p$), 
la composante locale en $q$ de $\Pi(w)$ est non ramifiÈe,
ses paramËtres de Satake aux places $v$ au dessus de $q$ (resp et distinctes de $w$)
Ètant donnÈs par $S_{\mathfrak m}(v)$;

\item la reprÈsentation $\Pi(w)_w$
est de la forme $\st_{i+2}(\chi_{w,0}) \times \chi_{w,1} \times \cdots \times \chi_{w,d-i-2}$, o˘ 
$\chi_{w,0},\cdots,\chi_{w,d-i-2}$ sont des caractËres non ramifiÈs de $F_w$.
\end{itemize}
\end{coro}

\begin{proof}
ConsidÈrons un sous-$\Tm_I$-module irrÈductible $M$ de $H^{-i}(X_{I,\bar \eta},V_\xi)_{\mathfrak m,tor}$.
Pour toute place $w \in \Spl(I)$, d'aprËs le changement de base lisse, on a 
$H^{-i}(X_{I,\bar \eta},V_\xi)_{\mathfrak m} \simeq H^{-i}(X^{\geq 1}_{I,\bar s_v},V_\xi)_{\mathfrak m}$.
D'aprËs la proposition prÈcÈdente ce module $M$ vÈrifie la propriÈtÈ \textbf{(P)} en Ètant de type $i+2$
de sorte qu'il existe une reprÈsentation automorphe irrÈductible $\xi$-cohomologique
$\Pi(w)$ telle que:
\begin{itemize}
\item sa composante locale en $w$ est de la forme 
$\st_{i+2}(\chi_{w,0}) \times \chi_{w,1} \times \cdots \times \chi_{w,d-i-2}$, o˘ 
$\chi_{w,0},\cdots,\chi_{w,d-i-2}$ sont des caractËres non ramifiÈs de $F_w$;

\item $M$ est un sous-quotient de la rÈduction modulo $l$ de $\Pi(w)$ et donc en particulier
pour toute place $v$ diffÈrente de $w$, au dessus d'un $p \in \Spl(I)$, sa composante locale en $v$
est non ramifiÈe avec pour paramËtres de Satake les ÈlÈments de $S_{\mathfrak m}(v)$.
\end{itemize}
\end{proof}

\subsection{SynthËse}

Fixons  
\begin{itemize}
\item une reprÈsentation irrÈductible algÈbrique $\xi$ de $G$ de dimension finie,

\item une place $v \in \Spl$ et une $\overline \Fm_l$-reprÈsentation supercuspidale $\varrho$
de $GL_{g_{-1}}(F_{v})$ pour $1 \leq g_{-1} \leq d$ telle que 
\begin{itemize}
\item $s_\xi(\varrho)=\lfloor \frac{d}{g_{-1}} \rfloor \geq 4$ et

\item $m(\varrho)=2$.
\end{itemize}

\item On choisit alors un niveau fini $I$ tel qu'il existe une reprÈsentation automorphe irrÈductible 
$\xi$-cohomologique $\Pi$ ayant des vecteurs non nuls invariants sous $I$ et dont la composante locale 
en $v$ est de la forme $\speh_{s_\xi(\varrho)}(\pi_{v}) \times ?$ o˘ $\pi_{v}$ est de $\varrho$-type $-1$.
\end{itemize}

Soit alors $\mathfrak m$ l'idÈal maximal de $\Tm_I$ dÈfini par $\Pi$ de sorte que,
d'aprËs le thÈorËme \ref{theo-equivalence2}, il existe
un indice $i$ pour lequel la torsion de $H^i(X_{I,\bar \eta},V_\xi)_{\mathfrak m}$ est non nulle.
Il dÈcoule alors du corollaire \ref{coro-principal} qu'il
existe une collection $\Pi(w)$ indexÈe par les places $w$ de $F$ divisant  une place fixÈe
au dessus d'un premier $p \in \Spl(I)$, de reprÈsentations irrÈductibles automorphes $\xi$-cohomologiques
deux ‡ deux non isomorphes entre elles et telles que:
\begin{itemize}
\item pour tout $q \in \Spl(I)$ distinct de $p$ (resp. $q=p$), 
la composante locale en $q$ de $\Pi(w)$ est non ramifiÈe,
ses paramËtres de Satake aux places $v'$ au dessus de $q$ (resp. et distinctes de $w$)
Ètant donnÈs par $S_{\mathfrak m}(v')$;

\item la reprÈsentation $\Pi(w)_w$
est de la forme $\st_{i+2}(\chi_{w,0}) \times \chi_{w,1} \times \cdots \times \chi_{w,d-i-2}$, o˘ 
$\chi_{w,0},\cdots,\chi_{w,d-i-2}$ sont des caractËres non ramifiÈs de $F_w$.
\end{itemize}

\bibliographystyle{plain}
\bibliography{bib-ok}

\def\cftil#1{\ifmmode\setbox7\hbox{$\accent"5E#1$}\else
  \setbox7\hbox{\accent"5E#1}\penalty 10000\relax\fi\raise 1\ht7
  \hbox{\lower1.15ex\hbox to 1\wd7{\hss\accent"7E\hss}}\penalty 10000
  \hskip-1\wd7\penalty 10000\box7} \def\cprime{$'$}
\begin{thebibliography}{10}

\bibitem{BBD}
J.~Bernstein, A.A. Beilinson, and P.~Deligne.
\newblock Faisceaux pervers.
\newblock In {\em Analyse et topologie sur les espaces singuliers, Asterisque
  100}, 1982.

\bibitem{boyer-invent}
P.~Boyer.
\newblock Mauvaise rÈduction des variÈtÈs de {D}rinfeld et correspondance de
  {L}anglands locale.
\newblock {\em Invent. Math.}, 138(3):573--629, 1999.

\bibitem{boyer-invent2}
P.~Boyer.
\newblock Monodromie du faisceau pervers des cycles \'evanescents de quelques
  vari\'et\'es de {S}himura simples.
\newblock {\em Invent. Math.}, 177(2):239--280, 2009.

\bibitem{boyer-compositio}
P.~Boyer.
\newblock Cohomologie des systËmes locaux de {H}arris-{T}aylor et applications.
\newblock {\em Compositio}, 146(2):367--403, 2010.

\bibitem{boyer-duke}
P.~Boyer.
\newblock La cohomologie des espaces de {L}ubin-{T}ate est libre.
\newblock {\em soumis}, 2013.

\bibitem{boyer-torsion}
P.~Boyer.
\newblock Filtrations de stratification de quelques variÈtÈs de shimura
  simples.
\newblock {\em Bulletin de la SMF}, 142, fascicule 4:777--814, 2014.

\bibitem{boyer-aif}
P.~Boyer.
\newblock Congruences automorphes et torsion dans la cohomologie d'un systËme
  local d'{H}arris-{T}aylor.
\newblock {\em Annales de l'Institut Fourier}, 65 n∞4:1669--1710, 2015.

\bibitem{boyer-entier}
P.~Boyer.
\newblock Faisceaux pervers entiers d'harris-taylor.
\newblock {\em preprint}, 2015.

\bibitem{boyer-imj}
P.~Boyer.
\newblock Sur la torsion dans la cohomologie des variÈtÈs de {S}himura de
  {K}ottwitz-{H}arris-{T}aylor.
\newblock {\em preprint}, 2015.

\bibitem{scholze-cara}
A.~Caraiani and P.~Scholze.
\newblock On the generic part of the cohomology of compact unitary shimura
  varieties.
\newblock {\em Preprint}, 2015.

\bibitem{dat-jl}
J.-F. Dat.
\newblock Un cas simple de correspondance de {J}acquet-{L}anglands modulo $l$.
\newblock {\em Proc. London Math. Soc. 104}, pages 690--727, 2012.

\bibitem{h-t}
M.~Harris, R.~Taylor.
\newblock {\em The geometry and cohomology of some simple {S}himura varieties},
  volume 151 of {\em Annals of Mathematics Studies}.
\newblock Princeton University Press, Princeton, NJ, 2001.

\bibitem{ito2}
T.~Ito.
\newblock Hasse invariants for somme unitary {S}himura varieties.
\newblock {\em Math. Forsch. Oberwolfach report 28/2005}, pages 1565--1568,
  2005.

\bibitem{lan-suh}
K.-W. Lan and J.~Suh.
\newblock {V}anishing theorems for torsion automorphic sheaves on compact
  {PEL}-type {S}himura varieties.
\newblock {\em Duke Math.}, 161(6):951--1170, 2012.

\bibitem{vigneras-induced}
M.-F. Vign{\'e}ras.
\newblock Induced {$R$}-representations of {$p$}-adic reductive groups.
\newblock {\em Selecta Math. (N.S.)}, 4(4):549--623, 1998.

\bibitem{vigneras-langlands}
M.-F. Vign{\'e}ras.
\newblock Correspondance de {L}anglands semi-simple pour {${\rm GL}(n,F)$}
  modulo {$l \neq p$}.
\newblock {\em Invent. Math.}, 144(1):177--223, 2001.

\bibitem{zelevinski2}
A.~V. Zelevinsky.
\newblock Induced representations of reductive {${p}$}-adic groups. {II}. {O}n
  irreducible representations of {${\rm GL}(n)$}.
\newblock {\em Ann. Sci. \'Ecole Norm. Sup. (4)}, 13(2):165--210, 1980.

\end{thebibliography}

\end{document}